\documentclass[10pt,pdftex]{amsart}

\usepackage{multirow}

\usepackage{amsmath,array,rotating,color,hyperref,appendix,tabularx}
\usepackage[a4paper,margin=1.8cm]{geometry}

\usepackage{tocvsec2}
\usepackage{booktabs}
\usepackage{float}

\newcommand{\Mat}[1]{\mathrm{Mat}_{#1}}
\newcommand{\diagtext}{\mathrm{diag}}
\newcommand{\blocktext}{\mathrm{block}}
\newcommand{\diag}[2]{\diagtext_{#1,#2}}
\newcommand{\block}[2]{\blocktext_{#1,#2}}
\newcommand{\coniugioraw}{\mathrm{C}}
\newcommand{\coniugiobaseraw}{\mathrm{D}}
\newcommand{\coniugiobase}[1]{\coniugiobaseraw_{#1}}
\newcommand{\coniugio}[2]{\coniugiobaseraw_{#1,#2}}
\newcommand{\coniugiototale}[1]{\coniugioraw_{#1}}

\newcommand{\conjblock}[1]{\coniugiototale{#1}}

\newcommand{\Spin}[1]{\ensuremath{\text{\upshape\rmfamily Spin}(#1)}}
\newcommand{\Spinnoarg}{\ensuremath{\text{\upshape\rmfamily Spin}}}
\newcommand{\SO}[1]{\mathrm{SO}(#1)}
\newcommand{\ug}{\;\shortstack{{\tiny\upshape def}\\=}\;}

\newcommand{\mathematica}{\textit{Mathematica}}
\newcommand{\I}{\mathcal{I}} 
\newcommand{\J}{J} 

\newcommand{\LC}{L^\CC}
\newcommand{\RH}{R^\HH}
\newcommand{\LH}{L^\HH}
\newcommand{\RO}{R^\OO}
\newcommand{\LO}{L^\OO}
\newcommand{\LF}{{L}}

\newcommand{\liespin}[1]{\mathop{\mathfrak{spin}}(#1)}
\newcommand{\lieso}[1]{\mathop{\mathfrak{so}}(#1)}

\newcommand{\Id}{\mathop{\mathrm{Id}}}
\newcommand{\Idarg}[1]{\mathop{\mathrm{Id}_{#1}}}
\newcommand{\End}[1]{\mathrm{End}(#1)}
\newcommand{\Cl}[1]{\mathrm{Cl}(#1)}

\newcommand{\norm}[1]{\Vert #1\Vert}

 \newcommand{\CC}{\mathbb{C}}   
\newcommand{\HH}{\mathbb{H}}   
\newcommand{\RR}{\mathbb{R}}   
 
\newcommand{\NN}{\mathbb{N}}

\newcommand{\OO}{\mathbb{O}}

\renewcommand{\SS}{\mathbb{S}}

\newcommand{\CP}[1]{\mathbb{C}P^{#1}}
\newcommand{\HP}[1]{\mathbb{H}P^{#1}}
\newcommand{\II}{\mathrm{Im}}

\numberwithin{equation}{section}

\newtheorem{te}{Theorem}[section]
\newtheorem*{te*}{Theorem}
\newtheorem{pr}[te]{Proposition}
\newtheorem{co}[te]{Corollary}
\newtheorem{lm}[te]{Lemma}

\theoremstyle{definition}
\newtheorem{de}[te]{Definition}    

\theoremstyle{remark}
\newtheorem{re}[te]{Remark}

\begin{document}

\title[Spheres with more than 7 vector fields: all the fault of $\Spin{9}$]{Spheres with more than 7 vector fields: \\ all the fault of $\Spin{9}$} 

\thanks{Both authors were supported by the MIUR under the PRIN Project ``Geometria Differenziale e Analisi Globale''}
\date{\today}

\author{Maurizio Parton}
\address{Universit\`a di Chieti-Pescara\\ Dipartimento di Scienze, viale Pindaro 87, I-65127 Pescara, Italy}
\email{parton@sci.unich.it}
\author{Paolo Piccinni}
\address{Dipartimento di Matematica\\ Sapienza - Universit\`a di Roma \\
Piazzale Aldo Moro 2, I-00185, Roma, Italy 
}
\email{piccinni@mat.uniroma1.it}

\begin{abstract}
We give an interpretation of the maximal number of linearly independent vector fields on spheres in terms of the $\Spin{9}$ representation on $\RR^{16}$. This casts an insight on the role of $\Spin{9}$  as a subgroup of $\SO{16}$ on the existence of vector fields on spheres, parallel to the one played by complex, quaternionic and octonionic structures on $\RR^2$, $\RR^4$ and $\RR^8$, respectively.
\end{abstract} 

\maketitle
\tableofcontents

\settocdepth{chapter}
\settocdepth{section}

\section{Introduction}

The existence of a nowhere zero vector field on odd dimensional spheres $S^{2n-1} \subset \RR^{2n}$ is an elementary consequence of the identification $\RR^{2n} = \CC^n$ and of the action  of the complex imaginary unit $i$ on the normal vector field $N$. Similarly, on spheres $S^{4n-1} \subset \RR^{4n}$ and $S^{8n-1} \subset \RR^{8n}$, one gets $3$ and $7$ tangent orthonormal vector fields from the identification $\RR^{4n} = \HH^n$ and $\RR^{8n} = \OO^n$. Here the $3$ imaginary units $i,j,k$ of quaternions $\HH$ and the $7$ imaginary units $i,j,k,e,f,g,h$ of octonions $\OO$ are used. These numbers $1,3,7$ give a maximal system of linearly independent vector fields on $S^{m-1} \subset \RR^m$, provided the (even) dimension $m$ of the ambient space is not divisible by $16$. 

The maximal number $\sigma(m)$ of linearly independent vector fields on $S^{m-1}$ is expressed as
\[
\sigma(m) = 2^p + 8q -1\enspace,
\]
where $\sigma(m) + 1 = 2^p + 8q$ is the \emph{Hurwitz-Radon number}, referring to the decomposition
\begin{equation}\label{eq:dec}
m = (2k+1)2^p 16^q\enspace, \qquad \text{where } 0 \leq p \leq 3\enspace.
\end{equation}
See~\cite{hur}, \cite {ra} for the original Hurwitz-Radon proof, obtained in the framework of compositions of quadratic forms. See also~\cite{ec} for a simplified proof, using representation theory of finite groups. Next,~\cite{ad1}, \cite{ad2} and~\cite{ad3} contain the J.~F.~Adams' celebrated theorem stating that $\sigma(m)$ is maximal. Also,~\cite{th} is an overview on related problems,~\cite[Chapters 11 and 15]{hu} and~\cite[Chapter V]{ka} are standard references.

The much more recent paper \cite{og} contains a combinatorial construction of a maximal system of orthonormal vector fields on spheres and an updated bibliography on the subject. In~\cite{og} a method of construction based on \emph{permutations of coordinates} is developed, generating tangent vector fields by acting on the normal vector through suitable \emph{monomial matrices}, that is, permutations and reflections of the coordinates. We will also proceed through permutations of coordinates and monomial matrices, although our main point is, as suggested in the title, to point out the role of the group $\Spin{9}$ in all the dimensions $m$ that allow more than $7$ linearly independent vector fields on $S^{m-1}$.

In Table~\ref{somespheres} we list some of the lowest dimensional spheres $S^{m-1} \subset \RR^m$ admitting a maximal number $\sigma(m)>7$ of linearly independent vector fields.

\begin{table}[H]
\caption{Some spheres $S^{m-1}$ with more than $7$ vector fields}\label{somespheres}
\begin{center}
\begin{tabular}{|c||c|c|c|c|c|c|c|c|c|c|c|c|c|c|c|c|c|}
\hline
$m-1$ & $15$ & $31$ & $47$ & $63$ & $79$ & $95$ & $111$ & $127$ & \dots & $255$ & \dots & $511$ & \dots & $1023$ & \dots & $2047$ & \dots\\
\hline 
\hline
$\sigma(m)$ & $8$ & $9$ & $8$ & $11$ & $8$ & $9$ & $8$ & $15$ & \dots & $16$ & \dots & $17$ & \dots & $19$ & \dots & $23$ & \dots\\
\hline
\end{tabular}
\end{center}
\end{table}

The first of them is $S^{15} \subset \RR^{16}$, that turns out to be a homogeneous space of the Lie group $\Spin{9}$. The unique Hopf fibration related to octonions can be written (cf.~\cite{gl-wa-zi}) in either of the two ways
\begin{equation}\label{hopffibration}
S^{15} \stackrel{S^7}{\longrightarrow} S^8\enspace, \qquad \frac{\Spin{9}}{\Spin{7}}\stackrel{\frac{\Spin{8}}{\Spin{7}}}{\longrightarrow}\frac{\Spin{9}}{\Spin{8}}\enspace.
\end{equation}

Indeed, a construction of $8$ orthonormal tangent vector fields on $S^{15} \subset \RR^{16}$ from the $\Spinnoarg$ representation of $\Spin{9}$ has been our starting point, although it was not completely clear for a while how such a construction extends to the next significative case, namely $S^{511} \subset \RR^{2\cdot 16^2}$. Accordingly, in writing the present paper we chose to postpone the proofs referring to arbitrary dimension after dealing with some ``low dimensional'' spheres, i.e.\ $S^{15}$ up to $S^{511}$. 
However, Section~\ref{sec:generalcase}, which contains the proof of the main statements for arbitrary dimension, is independent of the previous sections.

We have to mention that the framework we are going to use comes from Riemannian geometry in dimension $16$, that often refers to both the division algebra $\OO$ of octonions and the Lie group $\Spin{9}$. Just to give a couple of examples, we quote the study of $\Spin{9}$ as a weak holonomy group on Riemannian manifolds $M^{16}$~\cite{fr}, and the construction of exotic manifolds in the Cayley hyperbolic setting~\cite{af}.

Thus, we will consider $\Spin{9}$ as a subgroup of the rotation group $\SO{16}$, acting on $\RR^{16}=\OO^2$. It is generated by the block transformations
\begin{equation}\label{eq:haintro}
\left(
\begin{array}{c}
x \\
y
\end{array}
\right)
\longrightarrow
\left(
\begin{array}{cc}
r & \RO_{\overline u} \\
\RO_u & -r
\end{array}
\right)
\left(
\begin{array}{c}
x \\
y
\end{array}
\right)\enspace,
\end{equation}
where $(x,y) \in \OO^2$, $(r,u) \in S^8 \subset \RR \times \OO = \RR^9$ and $\RO_u, \RO_{\overline u}$ are the right multiplication on the octonions by $u, \overline u$, respectively (cf.\ Section~\ref{preliminaries} and~\cite[page 288]{ha}).

This approach focuses on the set of the nine self-dual involutions
\[
\I_{1},\dots,\I_{9}: \RR^{16} \longrightarrow \RR^{16}\enspace,
\]
defined by the nine choices $(r,u) = (1,0), (0,1), (0,i), \dots (0,h)$ in Formula~\eqref{eq:haintro}. These involutions satisfy the condition 
\[
\I_{\alpha} \I_{\beta} = -\I_{\beta}\I_{\alpha}\enspace, \qquad 1\leq \alpha < \beta \leq 9\enspace,
\]
(see~\cite[pages 287--289]{ha} and~\cite{fr}, \cite{pp}),
so that the $36$ compositions $\I_{\alpha} \I_{\beta}$ are complex structures on $\RR^{16}$.


We will see how the eight complex structures $\J_{1},\dots,\J_{8}$ on $\RR^{16}$ defined by
\[
\J_{\alpha}\ug\I_{\alpha}\I_{9}: \RR^{16} \longrightarrow \RR^{16}\enspace, \qquad \alpha =1,\dots,8
\]
play for our vector fields problem the same role as that of the action of the one complex, three quaternionic and seven octonionic units on $\RR^2$, $\RR^4$ and $\RR^8$, respectively. Indeed, we use these four kind of actions as fundamental ingredients in generating a set of vector fields on spheres of any dimension, see Theorems~\ref{teo:mainq}, \ref{teo:mainpq} and~\ref{teo:mainkpq}. To our knowledge, this role of \Spin{9} was never observed before.

Our results are briefly collected in Table~\ref{table:generalcase}, where $\conjblock{t}$ and $\coniugioraw$ are linear operators defined in Remark~\ref{re:overload}, and $\LF_i,\dots,\LF_h$ are left multiplications. It is worth remarking that, as in the permutation of coordinates method, $\conjblock{t}(\J_\cdot)$ and $\coniugioraw(\LF_\cdot)$ are monomial matrices.

\renewcommand{\arraystretch}{1.2}
\begin{table}[H]
\caption{A maximal system of vector fields on $S^{m-1}$}
{($m=(2k+1)2^p 16^q$, $k \geq 0$, $p=0,1,2,3$ and $q \geq 1$)} 
\label{table:generalcase}
\begin{center}
\begin{tabular}{|c||c|c|c|c|}
\hline
$(k,p,q)$ & Sphere & $\sigma(m)$ & Vector fields \\
\hline 
\hline
$(k,0,q)$ & $S^{(2k+1)16^q-1}$ & $8q$ & $\{\conjblock{t}(\J_{\alpha})\}_{\substack{t=1,\dots,q\\\alpha=1,\dots,8}}$ \\
\hline
$(k,1,q)$ & $S^{2(2k+1)16^q-1}$ & $8q+1$ & $\{\conjblock{t}(\J_{\alpha})\}_{\substack{t=1,\dots,q\\\alpha=1,\dots,8}}$ \\
& & & $\coniugioraw(\LF_i)$ \\
\hline
$(k,2,q)$ & $S^{4(2k+1)16^q -1}$ & $8q+3$ & $\{\conjblock{t}(\J_{\alpha})\}_{\substack{t=1,\dots,q\\\alpha=1,\dots,8}}$ \\
& & & $\coniugioraw(\LF_i),\coniugioraw(\LF_j),\coniugioraw(\LF_k)$ \\
\hline
$(k,3,q)$ & $S^{8(2k+1)16^q -1}$ & $8q+7$ &$\{\conjblock{t}(\J_{\alpha})\}_{\substack{t=1,\dots,q\\\alpha=1,\dots,8}}$ \\
& & & $\coniugioraw(\LF_i),\dots,\coniugioraw(\LF_h)$ \\
\hline
\end{tabular}
\end{center}
{For simplicity, we write $\conjblock{t}(\J)$ for $\conjblock{t}(\J)N$ and $\coniugioraw(\LF)$ for $\coniugioraw(\LF)N$, where $N$ is a normal unit vector field on $S^{m-1}$}
\end{table}
\renewcommand{\arraystretch}{1}

The computations for $S^{255}$ and up to $S^{8191}$ were first made with the help of the software \mathematica, which was the heuristic tool to formulate the correct form of the conjectures that became Theorems~\ref{teo:mainq}, \ref{teo:mainpq} and \ref{teo:mainkpq}.

Note that in our construction the sphere $S^{15}\subset\RR^{16}$ plays a basic role. Indeed, $S^{15}$ is the lowest dimensional sphere which admits more than $7$ tangent orthonormal vector fields. Also, $S^{15}$ is the total space of the three Hopf fibrations 
\[
S^{15} \stackrel{S^1}{\longrightarrow} \CP{7}\enspace,\qquad S^{15} \stackrel{S^3}{\longrightarrow}\HP{3}\enspace,\qquad S^{15} \stackrel{S^7}{\longrightarrow} S^8\enspace.
\]
Recall that the first two of them are not subfibrations of the third \cite{lv}. However, by writing down the vector fields tangent to the fibers in the three cases, one sees that no combination of them allows to get the maximal number $8$ of the orthonormal tangent vector fields on $S^{15}$. Table~\ref{table:generalcase} shows how the responsibility of such a maximal system on $S^{15}$ can be ascribed to the $\Spin{9}$ structure of $\RR^{16}$, and more generally it shows also how the same Lie group $\Spin{9}$ produces $q-1$ further $8$-ples of orthonormal vector fields when the dimension of the ambient space contains a factor $16^q$.

On the other hand, one can observe that the space of complex structures on $\RR^{16}$ splits, under the $\Spin{9}$ action:
\[
\Lambda^2(\RR^{16})=\Lambda^2_{36} \oplus \Lambda^2_{84} = \liespin{9} \oplus \Lambda^2_{84}\enspace,
\]  
and that our use in Table~\ref{table:generalcase} of the particular complex structures $\J_{1},\dots, \J_{8}$ on $\RR^{16}$ is just a possible choice, among the many ones, by suitable selections of $8$ complex structures in the component $\liespin{9} \subset \Lambda^2(\RR^{16})$, the Lie algebra of $\Spin{9}$.


In Section~\ref{preliminaries} we introduce specific notations to deal with low-dimensional cases. In Section~\ref{sec:15} we explain the $S^{15}$ situation as starting point for higher dimensions. In Section~\ref{sec:31-255} we show how  $S^{31}$, $S^{63}$, $S^{127}$ and $S^{255}$ can be seen in this respect as a combination of what obtained on $S^{15}$ and of the standard actions of $\CC,\HH,\OO$. Section~\ref{sec:511} introduces an iterative construction associated with the decomposition~\ref{eq:dec} of the dimension of the sphere in the most elementary case, that is, $S^{511}$. In Section~\ref{sec:generalcase} we introduce the general notation, then we state and prove our main statements, Theorems~\ref{teo:mainq}, \ref{teo:mainpq} and \ref{teo:mainkpq}.

\emph{Acknowledgements.} The authors wish to thank Rosa Gini for her help in developing the argument in Section~\ref{sec:generalcase}, and the referee for the careful reading of a first draft and for useful comments that lead us to revised proofs of the main statements.

\section{Preliminaries}\label{preliminaries}

We briefly recall the Cayley-Dickson process, used to construct new algebras from old ones. Let $\mathcal A$ be a \emph{$*$-algebra}, namely a real algebra equipped with a linear map $*:\mathcal A \rightarrow \mathcal A$, called \emph{conjugation}, satisfying
\[
a^{**}=a\enspace, \qquad (ab)^* = b^*a^*
\]
for all $a,b \in \mathcal A$. Then a new \emph{$*$-algebra} $\mathcal A'$ is defined by
\[
\mathcal A' \ug \{(a,b); a,b \in \mathcal A\}\enspace, \qquad (a,b)(c,d)\ug(ac-d^*b,da+bc^*)\quad\text{and}\quad
(a,b)^* \ug (a^*,-b)\enspace.
\]
This construction produces the algebra $\CC$ from $\RR$ (the linear map $*$ on the latter being the identity), then $\HH$ from $\CC$ and $\OO$ from $\HH$. We choose the standard canonical bases $\{1,i\}$, $\{1,i,j,k\}$ and $\{1,i,j,k,e,f,g,h\}$ for $\CC$, $\HH$ and $\OO$ respectively. In particular, we use the following multiplication table in $\OO$, where the left factor is in the first column.
\begin{table}[H]
\caption{Multiplication in $\OO$}
\centering
\begin{tabular}{|c||r|r|r|r|r|r|r|r|}
\hline 
$1$ & $i$ & $j$ & $k$ & $e$ & $f$ & $g$ & $h$\\
\hline 
\hline
$i$ & $-1$ & $k$ & $-j$ & $f$ & $-e$ & $-h$ & $g$\\
\hline 
$j$ & $-k$ & $-1$ & $i$ & $g$ & $h$ & $-e$ & $-f$\\
\hline
$k$ & $j$ & $-i$ & $-1$ & $h$ & $-g$ & $f$ & $-e$\\
\hline
$e$ & $-f$ & $-g$ & $-h$ & $-1$ & $i$ & $j$ & $k$\\
\hline
$f$ & $e$ & $-h$ & $g$ & $-i$ & $-1$ & $-k$ & $j$\\
\hline
$g$ & $h$ & $e$ & $-f$ & $-j$ & $k$ & $-1$ & $-i$\\
\hline
$h$ & $-g$ & $f$ & $e$ & $-k$ & $-j$ & $i$ & $-1$\\
\hline
\end{tabular}
\end{table}

In accordance with the Cayley-Dickson process, the multiplication of elements $x=h_1+h_2 e$, $y=k_1+k_2e\in\OO$ can be also viewed through the multiplication and the conjugation in $\HH$ by the formula
\begin{equation}\label{oct}
xy=(h_1k_1 -\overline k_2 h_2) + (k_2 h_1 + h_2 \overline k_1)e\enspace.
\end{equation}
Note also that the conjugation in $\OO$, defined as $\overline{x}\ug\overline h_1-h_2e$, gives the non-commutativity law $\overline{x y}=\bar y\bar x$.

\begin{lm}\label{lem:cayley-dickson}
Let $\mathcal A^n$ be the Cayley-Dickson algebra obtained inductively from $\mathcal A=\RR$ through the Cayley-Dickson process. Denote by $a^*$ and by $\Re(a)\ug\frac{1}{2}(a+a^*)$ the conjugate and the real part of elements $a \in \mathcal A^n$, respectively. Denote then by $[a,b,c]\ug(ab)c-a(bc)$ the associator of $a,b,c \in \mathcal A^n$, and by $<a,b>$ the scalar product in $\mathcal A^n = \RR^{2^n}$. Then the following formulas hold good for all $a,b,c \in \mathcal A^n$:
\begin{equation}\label{cd} 
(ab)^* = b^*a^*\enspace, \qquad \Re([a,b,c])=0\enspace, \qquad <a,b> = \Re(a b^*)\enspace.
\end{equation}
\end{lm}

\begin{proof}
The first and third formulas are verified by induction on $n$, the number of steps in the Cayley-Dickson process. 

To check the second formula, note that by linearity one can assume any of $a,b,c$ to be of the form $(x,0)$ or $(0,x)$, with $x \in \mathcal A^{n-1}$. Thus there are eight cases to be verified. Four of these cases lead to both products $(ab)c, a(bc)$ of type $(0,x)$, hence with zero real part. The case $a=(x,0), b=(y,0), c=(z,0)$ is done by induction on $n$. The remaining three cases are when two among $a,b,c$ are of the form $(x,0)$ and the third of the form $(0,x)$. For example, when $a=(x,0), b=(0,y), c=(0,z)$ one gets $\Re([a,b,c]) = \Re (x(z^*y)-z^*(yx)) = \Re ((z^*y)x-z^*(yx))$ and an inductive argument gives the conclusion.
\end{proof}


We denote by $\LH,\RH$ and $\LO,\RO$ the left and right multiplication in $\HH$ and $\OO$, respectively.
Explicit matrix representations for the right and left multiplication by $i,j,k$ in $\HH$ are
\begin{align}
\RH_i&= \left(
\begin{array}{rrrr}
0&-1&0&0\\
1&0& 0&0\\
0&0&0&1\\
0&0&-1&0
\end{array}
\right)\enspace,\qquad
\RH_j =\left(
\begin{array}{rrrr}
0&0&-1&0\\
0&0& 0&-1\\
1&0&0&0\\
0&1&0&0 
\end{array}
\right)\enspace,\qquad
\RH_k=\left(
\begin{array}{rrrr}
0&0&0&-1\\
0&0& 1&0\\
0&-1&0&0\\
1&0&0&0 
\end{array}
\right)\enspace,\label{right}\\
\LH_i&= \left(
\begin{array}{rrrr}
0&-1&0&0\\
1&0& 0&0\\
0&0&0&-1\\
0&0&1&0
\end{array}
\right)\enspace, \qquad
\LH_j =\left(
\begin{array}{rrrr}
0&0&-1&0\\
0&0& 0&1\\
1&0&0&0\\
0&-1&0&0 
\end{array}
\right)\enspace, \qquad
\LH_k=\left(
\begin{array}{rrrr}
0&0&0&-1\\
0&0& -1&0\\
0&1&0&0\\
1&0&0&0 
\end{array}
\right)\enspace.\label{left}
\end{align}

Although, as well-known, $\CC, \HH$ and $\OO$ are the only normed algebras over $\RR$,
we will use for our first examples also the algebra $\SS$ of sedenions, obtained from $\OO$ through the Cayley-Dickson process. Some characterization of $\SS$ has been given very recently in the context of \emph{locally complex algebras}, cf.~\cite{bss}. For further informations on $\SS$, see also~\cite{bdi} and~\cite{bcdi}.

Denoting by $1, e_1,\dots,e_{15}$ the canonical basis of $\SS$ over $\RR$, we can write the multiplication table~\ref{table:sedenions}, where it appears the existence of divisors of zeroes in $\SS$: for example $(e_2 - e_{11})(e_7 + e_{14})=0$.

\begin{table}[H]
\caption{Multiplication in $\SS$}\label{table:sedenions}
\begin{center}
\begin{tabular}{|c||r|r|r|r|r|r|r|r|r|r|r|r|r|r|r|r|}
\hline 
$1$ & $e_1$ & $e_2$ & $e_3$ & $e_4$ & $e_5$ & $e_6$ & $e_7$ & $e_8$ & $e_9$ & $e_{10}$ & $e_{11}$ & $e_{12}$ & $e_{13}$ & $e_{14}$ & $e_{15}$\\
\hline
\hline
$e_1$ & $-1$ & $e_3$ & $-e_2$ & $e_5$ & $-e_4$ & $-e_7$ & $e_6$ & $e_9$ & $-e_8$ & $-e_{11}$ & $e_{10}$ & $-e_{13}$ & $e_{12}$ & $e_{15}$ & $-e_{14}$\\
\hline 
$e_2$ & $-e_3$ & $-1$ & $e_1$ & $e_6$ & $e_7$ & $-e_4$ & $-e_5$ & $e_{10}$ & $e_{11}$ & $-e_8$ & $-e_9$ & $-e_{14}$ & $-e_{15}$ & $e_{12}$ & $e_{13}$\\
\hline
$e_3$ & $e_2$ & $-e_1$ & $-1$ & $e_7$ & $-e_6$ & $e_5$ & $-e_4$ & $e_{11}$ & $-e_{10}$ & $e_9$ & $-e_8$ & $-e_{15}$ & $e_{14}$ & $-e_{13}$ & $e_{12}$\\
\hline
$e_4$ & $-e_5$ & $-e_6$ & $-e_7$ & $-1$ & $e_1$ & $e_2$ & $e_3$ & $e_{12}$ & $e_{13}$ & $e_{14}$ & $e_{15}$ & $-e_8$ & $-e_9$ & $-e_{10}$ & $-e_{11}$\\
\hline
$e_5$ & $e_4$ & $-e_7$ & $e_6$ & $-e_1$ & $-1$ & $-e_3$ & $e_2$ & $e_{13}$ & $-e_{12}$ & $e_{15}$ & $-e_{14}$ & $e_9$ & $-e_8$ & $e_{11}$ & $-e_{10}$\\
\hline
$e_6$ & $e_7$ & $e_4$ & $-e_5$ & $-e_2$ & $e_3$ & $-1$ & $-e_1$ & $e_{14}$ & $-e_{15}$ & $-e_{12}$ & $e_{13}$ & $e_{10}$ & $-e_{11}$ & $-e_8$ & $e_9$\\
\hline
$e_7$ & $-e_6$ & $e_5$ & $e_4$ & $-e_3$ & $-e_2$ & $e_1$ & $-1$ & $e_{15}$ & $e_{14}$ & $-e_{13}$ & $-e_{12}$ & $e_{11}$ & $e_{10}$ & $-e_9$ & $-e_8$\\
\hline
$e_8$ & $-e_9$ & $-e_{10}$ & $-e_{11}$ & $-e_{12}$ & $-e_{13}$ & $-e_{14}$ & $-e_{15}$ & $-1$ & $e_1$ & $e_2$ & $e_3$ & $e_4$ & $e_5$ & $e_6$ & $e_7$\\
\hline
$e_9$ & $e_8$ & $-e_{11}$ & $e_{10}$ & $-e_{13}$ & $e_{12}$ & $e_{15}$ & $-e_{14}$ & $-e_1$ & $-1$ & $-e_3$ & $e_2$ & $-e_5$ & $e_4$ & $e_7$ & $-e_6$\\
\hline 
$e_{10}$ & $e_{11}$ & $e_8$ & $-e_9$ & $-e_{14}$ & $-e_{15}$ & $e_{12}$ & $e_{13}$ & $-e_2$ & $e_3$ & $-1$ & $-e_1$ & $-e_6$ & $-e_7$ & $e_4$ & $e_5$\\
\hline
$e_{11}$ & $-e_{10}$ & $e_9$ & $e_8$ & $-e_{15}$ & $e_{14}$ & $-e_{13}$ & $e_{12}$ & $-e_3$ & $-e_2$ & $e_1$ & $-1$ & $-e_7$ & $e_6$ & $-e_5$ & $e_4$\\
\hline
$e_{12}$ & $e_{13}$ & $e_{14}$ & $e_{15}$ & $e_8$ & $-e_9$ & $-e_{10}$ & $-e_{11}$ & $-e_4$ & $e_5$ & $e_6$ & $e_7$ & $-1$ & $-e_1$ & $-e_2$ & $-e_3$\\
\hline
$e_{13}$ & $-e_{12}$ & $e_{15}$ & $-e_{14}$ & $e_9$ & $e_8$ & $e_{11}$ & $-e_{10}$ & $-e_5$ & $-e_4$ & $e_7$ & $-e_6$ & $e_1$ & $-1$ & $e_3$ & $-e_2$\\
\hline
$e_{14}$ & $-e_{15}$ & $-e_{12}$ & $e_{13}$ & $e_{10}$ & $-e_{11}$ & $e_8$ & $e_9$ & $-e_6$ & $-e_7$ & $-e_4$ & $e_5$ & $e_2$ & $-e_3$ & $-1$ & $e_1$\\
\hline
$e_{15}$ & $e_{14}$ & $-e_{13}$ & $-e_{12}$ & $e_{11}$ & $e_{10}$ & $-e_9$ & $e_8$ & $-e_7$ & $e_6$ & $-e_5$ & $-e_4$ & $e_3$ & $e_2$ & $-e_1$ & $-1$\\
\hline
\end{tabular}
\end{center}
\end{table}

Consider now the sphere $S^{m-1} \subset\RR^m$, and decompose $m$ as $m=(2k+1)2^p16^q$, where $p\in\{0,1,2,3\}$. First, we observe that a vector field $B$ tangent to the sphere $S^{2^p16^q-1}\subset\RR^{2^p16^q}$ induces a vector field
\begin{equation}\label{eq:blockwise}
\underbrace{(B,\dots,B)}_{2k+1\text{ times}}
\end{equation}
tangent to the sphere $S^{(2k+1)2^p16^q-1}$. For this reason, we will assume in what follows that $m=2^p16^q$. Whenever we extend a vector field in this way, we call the vector field given by~\eqref{eq:blockwise} the \emph{diagonal extension} of $B$.

If $q=0$, that is, if $m$ is not divisible by $16$, the vector fields on $S^{m-1}$ are given by the complex, quaternionic or octonionic multiplication for $p=1,2$ or $3$ respectively. Hence, in this paper we are concerned with the case $q \geq 1$, that is, $m=16l$. In this case, it is convenient to denote the coordinates in $\RR^{16l}$ by $(s^1,\dots,s^l)$ where each $s^\alpha$, for $\alpha=1,\dots,l$, belongs to $\SS$, and can thus be identified with a pair $(x^\alpha,y^\alpha)$ of octonions.

The unit (outward) normal vector field $N$ of $S^{m-1}$ is still denoted by
\[
N \ug (s^1,\dots,s^l)\quad\text{where}\quad\norm{s^1}^2+\dots+\norm{s^l}^2=1\enspace.
\]
Therefore, we can think of $N$ as an element of $\SS^l=(\OO^2)^l=\RR^{16l}$.

Whenever $l=2,4$ or $8$, denote by $\coniugiobaseraw$ the automorphism of $\SS^l=(\OO^2)^l$ given by
\begin{equation}\label{star}
\coniugiobaseraw: ((x^1,y^1),\dots,(x^l,y^l))\longrightarrow ((x^1,-y^1),\dots, (x^l,-y^l))\enspace.
\end{equation}
We will refer to $\coniugiobaseraw$ as a \emph{conjugation}, due to its similarity with that in $*$-algebras.

Moreover, it is convenient to use formal notations as:
\begin{align}
N&=(s^1,s^2)\ug s^1+i s^2 \in S^{31}\enspace,\label{eq:formalC}\\
N&=(s^1,s^2,s^3,s^4)\ug s^1+i s^2 +j s^3 +k s^4 \in S^{63}\enspace,\label{eq:formalH}\\
N&=(s^1,s^2,s^3,s^4,s^5,s^6,s^7,s^8)\ug s^1+i s^2 +j s^3 +k s^4+e s^5+f s^6+g s^7+h s^8 \in S^{127}\enspace,\label{eq:formalO}
\end{align}
allowing to define left multiplications $\LF$ in sedenionic spaces $\SS^{2}$, $\SS^{4}$ and $\SS^{8}$ (like in $\CC$, $\HH$ and $\OO$) as follows.

If $l=2$ the left multiplication is
\begin{equation}\label{i}
\LF_i(s^1,s^2)\ug -s^2+is^1\enspace,
\end{equation}
whereas if $l=4$ we define
\begin{equation}\label{ijk}
\begin{split}
\LF_i(s^1,\dots,s^4)&\ug -s^2+i s^1-js^4+ k s^3\enspace,\\
\LF_j(s^1,\dots,s^4)&\ug -s^3+i s^4+js^1- k s^2\enspace,\\
\LF_k(s^1,\dots,s^4)&\ug -s^4-i s^3+js^2+ k s^1\enspace,
\end{split}
\end{equation}
and finally if $l=8$ we define
\begin{equation}\label{ijkefgh}
\begin{split}
\LF_i(s^1,\dots,s^8)&\ug -s^2+i s^1-js^4+ k s^3- e s^6+f s^5+gs^8-h s^7\enspace, \\
\LF_j(s^1,\dots,s^8)&\ug -s^3+i s^4+js^1- k s^2 - e s^7-f s^8 +g s^5+ h s^6\enspace, \\
\LF_k(s^1,\dots,s^8)&\ug -s^4-i s^3+js^2+ k s^1- e s^8+f s^7- g s^6+ h s^5\enspace, \\
\LF_e(s^1,\dots,s^8)&\ug -s^5+i s^6+js^7+ k s^6+e s^1-f s^2-gs^3 -h s^4\enspace, \\
\LF_f(s^1,\dots,s^8)&\ug -s^6-i s^5+js^8- k s^7+es^2+f s^1+gs^4-h s^3\enspace, \\
\LF_g(s^1,\dots,s^8)&\ug -s^7-i s^8-js^5+ k s^6+es^3-f s^4+gs^1+ h s^2\enspace, \\
\LF_h(s^1,\dots,s^8)&\ug -s^8+i s^7-js^6- k s^5 +es^4+f s^3-gs^2+ h s^1\enspace.
\end{split}
\end{equation}
Note that in all three cases $l=2,4$ and $8$ the vector fields $\LF_i(N),\dots,\LF_h(N)$ are tangent to $S^{31}$, $S^{63}$ and $S^{127}$, respectively.

In accordance with its role in the Hopf fibration~\eqref{hopffibration}, $\Spin{9}$ can be also defined as the subgroup of $\SO{16}$ preserving the decomposition of $\OO^2$ in octonionic lines $l_c=\{(x,cx)\}$ and $l_\infty =\{(0,y)\}$, where $x,c,y \in \OO$ (see~\cite{gl-wa-zi}). In the framework of $G$-structures, a \emph{$\Spin{9}$-structure} on a Riemannian  manifold $M^{16}$ can be defined as a rank $9$ vector subbundle $V \subset \End{TM}$, locally spanned by nine endomorphisms $\I_\alpha$ satisfying the following conditions:
\begin{equation}\label{top}
\I^2_\alpha = \Id\enspace, \qquad \I^*_\alpha = \I_\alpha\enspace, \qquad \I_\alpha \I_\beta = - \I_\beta \I_\alpha \quad \text{if}\quad \alpha \neq \beta\enspace,
\end{equation}
where $\I^*_\alpha$ denotes the adjoint of $\I_\alpha$ (cf.\ the Introduction as well as~\cite{fr}).

For $M=\RR^{16}$, $\I_1, \dots, \I_9$ are generators of the Clifford algebra $\Cl{9}$, considered as endomorphisms of its $16$-dimensional real representation $\Delta_9 = \RR^{16} = \OO^2$. Accordingly, unit vectors $v \in S^8 \subset \RR^9$ can be looked at, via the Clifford multiplication, as symmetric endomorphisms $v: \Delta_9 \rightarrow \Delta_9$. As seen in the Introduction, the explicit way to describe this is by $v = r + u \in S^8 \subset \RR\times\OO$ (that is, $r \in \RR, u \in \OO$ and $r^2+u\bar u=1$) acting on pairs $(x,y) \in \OO^2$ by
\begin{equation}\label{ha}
   \left(
 \begin{array}{c} 
x \\ 
y
 \end{array}
 \right)  \longrightarrow  \left( \begin{array}{cc} 
r& \RO_{\overline u} \\ 
\RO_u & -r
 \end{array}
 \right)  
 \left(
 \begin{array}{c} 
x \\ 
y
 \end{array}
 \right),
\end{equation}
where $\RO_u$ and $\RO_{\overline u}$ denote the right multiplication on the octonions by $u$ and $\overline u$, respectively  (cf.~\cite[page 280]{ha}).


A basis of the standard $\Spin{9}$ structure on $\SS=\OO^2$ can thus be written by looking at~\eqref{ha} and at the nine vectors $(0,1),(0,i),\dots,(0,h),(1,0) \in S^8 \subset \RR\times\OO$. It consists of the following symmetric endomorphisms:
\begin{equation}\label{topmatrices}
\begin{aligned}
\I_1&=\left(
\begin{array}{c|c}
0 & \Id \\ \hline
\Id & 0
\end{array}
\right)\enspace,\qquad &
\I_2&=\left(
\begin{array}{c|c}
0 & -\RO_i \\ \hline
\RO_i & 0
\end{array}
\right)\enspace,\qquad &
\I_3&=\left(
\begin{array}{c|c}
0 & -\RO_j \\ \hline
\RO_j & 0
\end{array}
\right)\enspace, \\
\I_4&=\left(
\begin{array}{c|c}
0 & -\RO_k \\ \hline
\RO_k & 0
\end{array}
\right)\enspace,\qquad &
\I_5&=\left(
\begin{array}{c|c}
0 & -\RO_e \\ \hline
\RO_e & 0
\end{array}
\right)\enspace,\qquad &
\I_6&=\left(
\begin{array}{c|c}
0 & -\RO_f\\ \hline
\RO_f & 0
\end{array}
\right)\enspace, \\
\I_7&=\left(
\begin{array}{c|c}
0 & -\RO_g \\ \hline
\RO_g & 0
\end{array}
\right)\enspace,\qquad &
\I_8&=\left(
\begin{array}{c|c}
0 & -\RO_h \\ \hline
\RO_h & 0
\end{array}
\right)\enspace,\qquad &
\I_9&=\left(
\begin{array}{c|c}
\Id & 0 \\ \hline
0 & -\Id
\end{array}
\right)\enspace.
\end{aligned}
\end{equation}

The right octonionic multiplications $\RO_i,\dots,\RO_h$ can be written as $8 \times 8$ matrices using~\eqref{oct}:
\begin{equation}\label{matricesspin7}
\begin{aligned}
\RO_i=\left(
\begin{array}{c|c}
\RH_i & 0 \\ \hline
0 & -\RH_i
\end{array}
\right)\enspace,\qquad
\RO_j&=\left(
\begin{array}{c|c}
\RH_j & 0 \\ \hline
0 & -\RH_j
\end{array}
\right)\enspace,\qquad &
\RO_k&=\left(
\begin{array}{c|c}
\RH_k & 0 \\ \hline
0 & -\RH_k
\end{array}
\right)\enspace, \\
\RO_e&=\left(
\begin{array}{c|c}
0 & -\Id \\ \hline
\Id & 0
\end{array}
\right)\enspace,\qquad &
\RO_f&=\left(
\begin{array}{c|c}
0 & \LH_i \\ \hline
\LH_i & 0
\end{array}
\right)\enspace,\\
\RO_g&=\left(
\begin{array}{c|c}
0 & \LH_j \\ \hline
\LH_j & 0
\end{array}
\right)\enspace,\qquad &
\RO_h&=\left(
\begin{array}{c|c}
0 & \LH_k \\ \hline
\LH_k & 0
\end{array}
\right)\enspace.
\end{aligned}
\end{equation}

The space $\Lambda^2\RR^{16}$ of $2$-forms in $\RR^{16}$ decomposes under $\Spin{9}$ as
\[
\Lambda^2\RR^{16} = \Lambda^2_{36} \oplus \Lambda^2_{84}
\]
(cf.~\cite[page 146]{fr}), where  $\Lambda^2_{36} \cong \liespin{9}$ and $ \Lambda^2_{84}$ is an orthogonal complement in $\Lambda^2\RR^{16} \cong \lieso{16}$. Explicit generators of both subspaces can be written by looking at the $36$ compositions $\J_{\alpha \beta} \ug \I_\alpha \I_\beta$, for $\alpha <\beta$, and at the $84$ compositions $\J_{\alpha \beta \gamma} \ug \I_\alpha \I_\beta \I_\gamma$, for $\alpha <\beta<\gamma$, all complex structures on $\RR^{16}$.

Among the $36$ complex structures $\J_{\alpha \beta}$, in this paper we use only the eight $\J_\alpha\ug\J_{\alpha 9}$, whose matrix form is:
\begin{equation}\label{eq:J2}
\begin{aligned}
\J_{1}&=\left(
\begin{array}{c|c}
0 & -\Id \\ \hline
\Id & 0
\end{array}
\right),\thinspace &
\J_{2}&=\left(
\begin{array}{c|c}
0 & \RO_i \\ \hline
\RO_i & 0
\end{array}
\right),\thinspace &
\J_{3}&=\left(
\begin{array}{c|c}
0 & \RO_j \\ \hline
\RO_j & 0
\end{array}
\right),\thinspace &
\J_{4}&=\left(
\begin{array}{c|c}
0 & \RO_k \\ \hline
\RO_k & 0
\end{array}
\right),\\
\J_{5}&=\left(
\begin{array}{c|c}
0 & \RO_e \\ \hline
\RO_e & 0
\end{array}
\right),\thinspace &
\J_{6}&=\left(
\begin{array}{c|c}
0 & \RO_f \\ \hline
\RO_f & 0
\end{array}
\right),\thinspace &
\J_{7}&=\left(
\begin{array}{c|c}
0 & \RO_g \\ \hline
\RO_g & 0
\end{array}
\right),\thinspace &
\J_{8}&=\left(
\begin{array}{c|c}
0 & \RO_h \\ \hline
\RO_h & 0
\end{array}
\right).
\end{aligned}
\end{equation}
For the matrices associated with the remaining $\J_{\alpha \beta}$, where $1\leq \alpha < \beta \leq 8$, see~\cite{pp}.

\begin{re}\label{re:blockwise}
Of course, one can see the complex structures $\J_{\alpha}$ as defined by \eqref{eq:J2} on the algebra $\SS$ of sedenions. In the following Sections~\ref{sec:15}, \ref{sec:31-255} and \ref{sec:511} we use the same symbol $\J_{\alpha}$ to denote their diagonal extensions to $\SS^l$, as in \eqref{eq:blockwise}, so that if $N=(s^1,\dots,s^l)$ we have $\J_{\alpha}N\ug(\J_{\alpha}s^1,\dots,\J_{\alpha}s^l)$.
\hfill\qed
\end{re}

\begin{re}
As a matter of notations, it is worth to mention that, throughout the paper, letters $i,j,k,e,f,g,h$ denote only units in $\CC$, $\HH$ and~$\OO$. Instead, indexes are denoted by greek letters $\alpha, \beta, \dots$. 
\hfill\qed
\end{re}

 
\section{The sphere $S^{15}$}\label{sec:15}

Denote by
\[
N\ug (x,y)\ug(x_1,\dots,x_8,y_1,\dots,y_8) 
\]
the (outward) unit normal vector field of $S^{15} \subset \RR^{16}=\SS$. We point out that the existence of zero divisors in the algebra $\SS$ of sedenions infers that left multiplications 
by the sedenions unities $e_\alpha$, for $\alpha=1,\dots, 15$, give rise to $15$ vector fields that are \emph{not linearly independent}. For example, consider the normal vector $\vec b = \frac{1}{\sqrt 2}  e_7 +\frac{1}{\sqrt 2}  e_{14}$: by looking at Table~\ref{table:sedenions} one sees that $e_2\vec b = e_{11}\vec b$. 

On the other hand, the identification $\RR^{16}=\OO^2$ gives on $S^{15}$ the $7$ tangent vector fields
\begin{equation}\label{hopf}
\begin{split}
\LO_iB\ug (\LO_ix,\LO_iy)&=(-x_2,x_1,-x_4,x_3,-x_6,x_5,x_8,-x_7,-y_2,y_1,-y_4,y_3,-y_6,y_5,y_8,-y_7)\enspace,\\
\LO_jB\ug (\LO_jx,\LO_jy)&=(-x_3,x_4,x_1,-x_2,-x_7,-x_8,x_5,x_6,-y_3,y_4,y_1,-y_2,-y_7,-y_8,y_5,y_6)\enspace,\\
\LO_kB\ug (\LO_kx,\LO_ky)&=(-x_4,-x_3,x_2,x_1,-x_8,x_7,-x_6,x_5,-y_4,-y_3,y_2,y_1,-y_8,y_7,-y_6,y_5)\enspace,\\
\LO_eB\ug (\LO_ex,\LO_ey)&=(-x_5,x_6,x_7,x_8,x_1,-x_2,-x_3,-x_4,-y_5,y_6,y_7,y_8,y_1,-y_2,-y_3,-y_4)\enspace,\\
\LO_fB\ug (\LO_fx,\LO_fy)&=(-x_6,-x_5,x_8,-x_7,x_2,x_1,x_4,-x_3,-y_6,-y_5,y_8,-y_7,y_2,y_1,y_4,-y_3)\enspace,\\
\LO_gB\ug (\LO_gx,\LO_gy)&=(-x_7,-x_8,-x_5,x_6,x_3,-x_4,x_1,x_2,-y_7,-y_8,-y_5,y_6,y_3,-y_4,y_1,y_2)\enspace,\\
\LO_hB\ug (\LO_hx,\LO_hy)&=(-x_8,x_7,-x_6,-x_5,x_4,x_3,-x_2,x_1,-y_8,y_7,-y_6,-y_5,y_4,y_3,-y_2,y_1)\enspace,\\
\end{split}
\end{equation}
spanning the distribution of vertical leaves in the Hopf fibration $S^{15} \rightarrow S^8$. Thus, any other vector field on $S^{15} \subset \RR^{16}$ orthogonal to the $7$ listed in~\eqref{hopf} should belong to the horizontal distribution of the Hopf fibration.


To work out a construction of $8$ orthonormal vector fields on $S^{15}$, consider (as a possible choice among the $36$ complex structures $\I_{\alpha}\I_{\beta}$, $\alpha < \beta$) the eight complex structures given in~\eqref{eq:J2}:
\[
\J_{1}, \dots, \J_{8}: \SS \longrightarrow \SS\enspace.
\] 

\begin{pr}
The $8$ vector fields
\begin{equation}\label{otto}
\begin{split}
\J_{1}N =&(-y_1,-y_2,-y_3,-y_4,-y_5,-y_6,-y_7,-y_8,x_1,x_2,x_3,x_4,x_5,x_6,x_7,x_8)\enspace,\\
\J_{2}N =&(-y_2,y_1,y_4,-y_3,y_6,-y_5,-y_8,y_7,-x_2,x_1,x_4,-x_3,x_6,-x_5,-x_8,x_7)\enspace,\\
\J_{3}N =&(-y_3,-y_4,y_1,y_2,y_7,y_8,-y_5,-y_6,-x_3,-x_4,x_1,x_2,x_7,x_8,-x_5,-x_6)\enspace,\\
\J_{4}N =&(-y_4,y_3,-y_2,y_1,y_8,-y_7,y_6,-y_5,-x_4,x_3,-x_2,x_1,x_8,-x_7,x_6,-x_5)\enspace,\\
\J_{5}N =&(-y_5,-y_6,-y_7,-y_8,y_1,y_2,y_3,y_4,-x_5,-x_6,-x_7,-x_8,x_1,x_2,x_3,x_4)\enspace,\\
\J_{6}N =&(-y_6,y_5,-y_8,y_7,-y_2,y_1,-y_4,y_3,-x_6,x_5,-x_8,x_7,-x_2,x_1,-x_4,x_3)\enspace,\\
\J_{7}N =&(-y_7,y_8,y_5,-y_6,-y_3,y_4,y_1,-y_2,-x_7,x_8,x_5,-x_6,-x_3,x_4,x_1,-x_2)\enspace,\\
\J_{8}N =&(-y_8,-y_7,y_6,y_5,-y_4,-y_3,y_2,y_1,-x_8,-x_7,x_6,x_5,-x_4,-x_3,x_2,x_1)
\end{split}
\end{equation}
are tangent to $S^{15}$ and orthonormal.
\end{pr}

This comes indeed as a special case of the following slightly more general proposition.
\begin{pr}
Fix any $\beta$, $1\leq \beta \leq 9$, and consider the $8$ complex structures $\I_\alpha\I_\beta$, with $\alpha \neq \beta$. Then the $8$ vector fields
$\I_\alpha\I_\beta N$ are tangent to $S^{15}$ and orthonormal.
\end{pr}

\begin{proof}
It is sufficient to use properties~\eqref{top} of the symmetric endomorphisms $\I_{\alpha}$. 
\end{proof}

\section{Spheres $S^{2^p16-1}$, for $p=1,2,3$, and $S^{255}$}\label{sec:31-255}

In this Section, we write explicitly maximal systems of vector fields for $S^{31}$, $S^{63}$, $S^{127}$. The proof that any of these systems is orthonormal is straightforward. We will also explain the $S^{255}$ case.

\subsection*{Case $\mathbf{p=1}$}

The next sphere having more than $7$ tangent vector fields is $S^{31}$, whose maximal number is $9$. In this case, we obtain $8$ vector fields by writing the unit normal vector field as $N=(s^1,s^2)=(x^1,y^1,x^2,y^2)\in S^{31}\subset \SS^{2}$, where $x^1,y^1,x^2,y^2\in\OO$, and repeating Formulas~\eqref{otto} for each pair $(x^1,y^1),(x^2,y^2)$:
\begin{equation}\label{otto'}
\begin{split}
\J_{1}N&=(\J_{1}s^1,\J_{1}s^2)\\
&=(-y^1_1,-y^1_2,\dots ,-y^1_7 ,-y^1_8,x^1_1,x^1_2,\dots ,x^1_7 ,x^1_8,-y^2_1,-y^2_2, \dots ,-y^2_7,-y^2_8,x^2_1,x^2_2,\dots ,x^2_7 ,x^2_8)\enspace,\\
\J_{2}N&=(\J_{2}s^1,\J_{2}s^2)\\
&=(-y^1_2,y^1_1,\dots ,-y^1_8,y^1_7,-x^1_2,x^1_1,\dots -x^1_8,x^1_7,-y^2_2,y^2_1,\dots ,-y^2_8 ,y^2_7,-x^2_2,x^2_1,\dots ,-x^2_8 ,x^2_7)\enspace,\\
\J_{3}N&=(\J_{3}s^1,\J_{3}s^2)\\
&=(-y^1_3,-y^1_4,\dots ,-y^1_5,-y^1_6,-x^1_3,-x^1_4,\dots ,-x^1_5,-x^1_6,-y^2_3,-y^2_4,\dots ,-y^2_5,-y^2_6,-x^2_3,-x^2_4,\dots ,-x^2_5,-x^2_6)\enspace,\\
\J_{4}N&=(\J_{4}s^1,\J_{4}s^2)\\
&=(-y^1_4,y^1_3,\dots ,y^1_6 ,-y^1_5,-x^1_4,x^1_3,\dots ,x^1_6 ,-x^1_5,-y^2_4,y^2_3,\dots ,y^2_6,-y^2_5,-x^2_4,x^2_3,\dots ,x^2_6,-x^2_5)\enspace,\\
\J_{5}N&=(\J_{5}s^1,\J_{5}s^2)\\
&=(-y^1_5,-y^1_6,\dots ,y^1_3 ,y^1_4,-x^1_5,-x^1_6,\dots ,x^1_3,x^1_4,-y^2_5,-y^2_6,\dots ,y^2_3,y^2_4,-x^2_5,-x^2_6,\dots ,x^2_3 ,x^2_4)\enspace,\\
\J_{6}N&=(\J_{6}s^1,\J_{6}s^2)\\
&=(-y^1_6,y^1_5,\dots ,-y^1_4 ,y^1_3,-x^1_6,x^1_5,\dots ,-x^1_4 ,x^1_3,-y^2_6,y^2_5,\dots ,-y^2_4 ,y^2_3,-x^2_6,x^2_5,\dots ,-x^2_4 ,x^2_3)\enspace,\\
\J_{7}N&=(\J_{7}s^1,\J_{7}s^2)\\
&=(-y^1_7,y^1_8,\dots ,y^1_1 ,-y^1_2,-x^1_7,x^1_8,\dots ,x^1_1 ,-x^1_2,-y^2_7,y^2_8,\dots ,y^2_1 ,-y^2_2,-x^2_7,x^2_8,\dots ,x^2_1,-x^2_2)\enspace,\\
\J_{8}N&=(\J_{8}s^1,\J_{8}s^2)\\
&=(-y^1_8,-y^1_7,\dots ,y^1_2 ,y^1_1,-x^1_8,-x^1_7, \dots ,x^1_2 ,x^1_1,-y^2_8,-y^2_7,\dots ,y^2_2 ,y^2_1,-x^2_8,-x^2_7, \dots ,x^2_2 ,x^2_1)\enspace.
\end{split}
\end{equation}

A ninth orthonormal vector field, completing the maximal system, is found using the formal left multiplication~\eqref{i} and the automorphism $\coniugiobaseraw$~\eqref{star}:
\begin{equation}
\coniugiobaseraw(\LF_i N)=\coniugiobaseraw(-s^2,s^1)=(-x^2,y^2,x^1,-y^1)\enspace.
\end{equation}

\subsection*{Case $\mathbf{p=2}$}

The sphere $S^{63}$ has a maximal number of $11$ orthonormal vector fields. The normal vector field is in this case given by $N=(s^1,\dots,s^4)=(x^1,y^1,\dots,x^4,y^4)\in S^{63} \subset \SS^{4}$, and $8$ vector fields arise as $\J_{\alpha}N$, for $\alpha=1,\dots,8$. Three other vector fields are again given by the formal left multiplications~\ref{ijk} and the automorphism $\coniugiobaseraw$~\eqref{star}:
\begin{equation}
\begin{split}
\coniugiobaseraw(\LF_i N) &=(-x^2,y^2,x^1,-y^1,-x^4,y^4,x^3,-y^3)\enspace,\\ 
\coniugiobaseraw (\LF_j N) &=(-x^3,y^3,x^4,-y^4,x^1,-y^1,-x^2,y^2)\enspace,\\
\coniugiobaseraw (\LF_k N) &=(-x^4,y^4,-x^3,y^3,x^2,-y^2,x^1,-y^1)\enspace.
\end{split}
\end{equation}

\subsection*{Case $\mathbf{p=3}$}

The sphere $S^{127}$ has a maximal number of $15$ orthonormal vector fields. Eight of them are still given by $\J_{\alpha}N$, for $\alpha=1,\dots,8$, whereas the formal left multiplications given in~\ref{ijkefgh} yield the $7$ tangent vector fields $\coniugiobaseraw(\LF_\alpha N)$, for $\alpha\in\{i,\dots,h\}$.


\begin{re}
Up to dimension $127$ the vector fields were built through the following construction: the first $8$ were the (diagonal extension of the) $\J_{\alpha}$, and the next $1,3$ or $7$ respectively obtained from the $\CC$, $\HH$ or $\OO$ actions by left multiplications on blocks of $16$ coordinates, as in Formulas~\eqref{i},~\eqref{ijk},~\eqref{ijkefgh}. We call those actions the \emph{block extensions} of the original actions, and in what follows we denote the extension of $A$ by $\blocktext(A)$.
\hfill\qed
\end{re}

\subsection*{The sphere $S^{255}$}\label{sec:256}

To write a system of $16$ orthonormal vector fields on $S^{255} \subset \RR^{256}$ consider the decomposition
\begin{equation}\label{256}
\RR^{256} = \RR^{16} \oplus \dots \oplus \RR^{16}\enspace,
\end{equation}
and observe that both the number and the dimension of components are sixteen. The unit outward normal vector field can be written as
\[
N=(s^1,\dots,s^{16})\enspace,
\]
where $s^1,\dots, s^{16}$ are sedenions.

Consider now the $16 \times 16$ matrices giving the complex structures $\J_{1},\dots,\J_{8}$, and listed as~\eqref{eq:J2}. They act on $N$ not only on the $16$-dimensional components of~\eqref{256}, but also formally on the (column) $16$-ples of sedenions $(s^1,\dots,s^{16})^T$. According to which of the two actions of the same matrices are considered in $\RR^{256}$, we use the notations 
\[
\J_{1},\dots,\J_{8} \qquad \text{or} \qquad \blocktext(\J_{1}),\dots, \blocktext(\J_{8})
\]
for the obtained complex structures on $\RR^{256}$. The following $16$ vector fields are obtained: 
\begin{align}
\J_{1}N\enspace,\enspace&\dots\enspace,\J_{8}N\enspace,\label{eq:level1}\\
\coniugiobaseraw(\blocktext(\J_{1})N)\enspace,\enspace&\dots\enspace,\coniugiobaseraw(\blocktext(\J_{8})N)\enspace,\label{eq:level2}
\end{align}
where $\coniugiobaseraw$ has been defined in Formula~\eqref{star}.
We call \emph{level $1$ vector fields} and \emph{level $2$ vector fields} the ones given by \eqref{eq:level1} and \eqref{eq:level2} respectively.

We checked by \mathematica\ that \eqref{eq:level1} and \eqref{eq:level2} give an orthonormal (maximal) system, but we give now an elementary algebraic proof.

\begin{pr}\label{16}
Formulas~\eqref{eq:level1} and \eqref{eq:level2} give a maximal system of $16$ orthonormal tangent vector fields on $S^{255}$.
\end{pr}

\begin{proof}
As in Section~\ref{preliminaries}, we denote sedenions as pairs $s^\alpha\ug(x^\alpha,y^\alpha)$ of octonions. 
The unit normal vector field is
\begin{equation}
N=(s^1,\dots,s^{16})=(x^1,y^1,\dots,x^{16},y^{16}) \in S^{255}\enspace,
\end{equation}
and one gets the tangent vectors:
\begin{equation}\label{B}
\begin{split}
\J_{1}N&=(\J_{1} s^1,\dots,\J_{1}s^{16}) = (-y^1,x^1,\dots,-y^{16},x^{16})\enspace,\\
\J_{2}N&=(\J_{2} s^1,\dots,\J_{2}s^{16}) = (\RO_i y^1,\RO_i x^1,\dots,\RO_i y^{16}, \RO_i x^{16})\enspace,\\
\J_{3}N&=(\J_{3} s^1,\dots,\J_{3}s^{16}) = (\RO_j y^1,\RO_j x^1,\dots,\RO_j y^{16}, \RO_j x^{16})\enspace,\\
\J_{4}N&=(\J_{4} s^1,\dots,\J_{4}s^{16}) = (\RO_k y^1,\RO_k x^1,\dots,\RO_k y^{16}, \RO_k x^{16})\enspace,\\
\J_{5}N&=(\J_{5} s^1,\dots,\J_{5}s^{16}) = (\RO_e y^1,\RO_e x^1,\dots,\RO_e y^{16}, \RO_e x^{16})\enspace,\\
\J_{6}N&=(\J_{6} s^1,\dots,\J_{6}s^{16}) = (\RO_f y^1,\RO_f x^1,\dots,\RO_f y^{16}, \RO_f x^{16})\enspace,\\
\J_{7}N&=(\J_{7} s^1,\dots,\J_{7}s^{16}) = (\RO_g y^1,\RO_g x^1,\dots,\RO_g y^{16}, \RO_g x^{16})\enspace,\\
\J_{8}N&=(\J_{8} s^1,\dots,\J_{8}s^{16}) = (\RO_h y^1,\RO_h x^1,\dots,\RO_h y^{16}, \RO_h x^{16})\enspace,
\end{split}
\end{equation}
that are easily checked to be orthonormal.

Moreover, one obtains eight further vector fields:
\begin{equation}\label{B'}
\begin{split}
\coniugiobaseraw(\blocktext(\J_{1})N)&=\coniugiobaseraw(-s^9,-s^{10},-s^{11},-s^{12},-s^{13},-s^{14},-s^{15},-s^{16}, s^1,s^2,s^3,s^4,s^5,s^6,s^7,s^8)\\
&= (-x^9,y^9,-x^{10},y^{10},-x^{11},y^{11},-x^{12},y^{12}, -x^{13},y^{13},-x^{14},y^{14},-x^{15}, y^{15},-x^{16},y^{16}, \\
&\phantom{=(-}x^1,-y^1,x^2,-y^2,x^3,-y^3,x^4,-y^4,x^5,-y^5,x^6,-y^6,x^7,-y^7,x^8,-y^8)\enspace,\\
\coniugiobaseraw(\blocktext(\J_{2})N)&=\coniugiobaseraw(-s^{10}, s^9,s^{12},-s^{11},s^{14},-s^{13}, -s^{16},s^{15}, -s^2,s^1,s^4,-s^3,s^6,-s^5,-s^8,s^7)\\
&=(-x^{10},y^{10},x^9,-y^9,x^{12},-y^{12},-x^{11}, y^{11},x^{14},-y^{14}, -x^{13},y^{13},-x^{16},y^{16},x^{15}, -y^{15}, \\
&\phantom{=(}-x^2, y^2, x^1,-y^1,x^4,-y^4,-x^3,y^3,x^6,-y^6,-x^5,y^5,-x^8,y^8,x^7,-y^7)\enspace,\\
\coniugiobaseraw(\blocktext(\J_{3})N)&= \coniugiobaseraw(-s^{11},-s^{12},s^{9},s^{10},s^{15},s^{16}, -s^{13},-s^{14},- s^3,-s^4,s^1,s^2,s^7,s^8,-s^5,-s^6)\\
&= (-x^{11},y^{11},-x^{12},y^{12},x^{9},-y^{9},x^{10},-y^{10}, x^{15},-y^{15},x^{16},-y^{16},-x^{13}, y^{13},-x^{14},y^{14}, \\
&\phantom{=(}-x^3,y^3,-x^4,y^4,x^1,-y^1,x^2,-y^2,x^7,-y^7,x^8,-y^8,-x^5,y^5,-x^6,y^6)\enspace,\\
\coniugiobaseraw(\blocktext_2(\J_{4})N)&=\coniugiobaseraw(-s^{12},s^{11},-s^{10},s^{9},s^{16},-s^{15}, s^{14},-s^{13}, -s^4,s^3,-s^2,s^1,s^8,-s^7,s^6,-s^5)\\
&= (-x^{12},y^{12},x^{11},-y^{11},-x^{10},y^{10},x^{9},-y^{9}, x^{16},-y^{16},-x^{15},y^{15},x^{14}, -y^{14},-x^{13},y^{13}, \\
&\phantom{=(}-x^4,y^4,x^3,-y^3,-x^2,y^2,x^1,-y^1,x^8,-y^8,-x^7,y^7,x^6,-y^6,-x^5,y^5)\enspace,\\
\coniugiobaseraw(\blocktext(\J_{5})N)&= \coniugiobaseraw(-s^{13},-s^{14},-s^{15},-s^{16},s^{9},s^{10}, s^{11},s^{12},- s^5,-s^6,-s^7,-s^8,s^1,s^2,s^3,s^4)\\
&= (-x^{13},y^{13},-x^{14},y^{14},-x^{15},y^{15},-x^{16},y^{16}, x^{9},-y^{9},x^{10},-y^{10},x^{11},- y^{11},x^{12},-y^{12}, \\
&\phantom{=(}-x^5,y^5,-x^6,y^6,-x^7,y^7,-x^8,y^8,x^1,-y^1,x^2,-y^2,x^3,-y^3,x^4,-y^4)\enspace,\\
\coniugiobaseraw(\blocktext(\J_{6})N)&= \coniugiobaseraw(-s^{14},s^{13},-s^{16},s^{15},-s^{10},s^{9}, -s^{12},s^{11}, -s^6,s^5,-s^8,s^7,-s^2,s^1,-s^4,s^3)\\
&= (-x^{14},y^{14},x^{13},-y^{13},-x^{16},y^{16},x^{15},-y^{15}, -x^{10},y^{10},x^{9},-y^{9},-x^{12}, y^{12},x^{11},-y^{11}, \\
&\phantom{=(}-x^6,y^6,x^5,-y^5,-x^8,y^8,x^7,-y^7,-x^2,y^2,x^1,-y^1,-x^4,y^4,x^3,-y^3)\enspace,\\
\coniugiobaseraw(\blocktext(\J_{7})N)&= \coniugiobaseraw(-s^{15},s^{16},s^{13},-s^{14},-s^{11},s^{12}, s^{9},-s^{10}, -s^7,s^8,s^5,-s^6,-s^3,s^4,s^1,-s^2)\\
&= (-x^{15},y^{15},x^{16},-y^{16},x^{13},-y^{13},-x^{14},y^{14}, -x^{11},y^{11},x^{12},-y^{12},x^{9}, -y^{9},-x^{10},y^{10}, \\
&\phantom{=(}-x^7,y^7,x^8,-y^8,x^5,-y^5,-x^6,y^6,-x^3,y^3,x^4,-y^4,x^1,-y^1,-x^2,y^2)\enspace,\\
\coniugiobaseraw(\blocktext(\J_{8})N)&= \coniugiobaseraw(-s^{16},-s^{15},s^{14},s^{13},-s^{12},-s^{11}, s^{10},s^{9}, -s^8,-s^7,s^6,s^5,-s^4,-s^3,s^2,s^1)\\
&= (-x^{16},y^{16},-x^{15},y^{15},x^{14},-y^{14},x^{13},-y^{13}, -x^{12},y^{12},-x^{11},y^{11},x^{10},- y^{10},x^{9},-y^{9}, \\
&\phantom{=(}-x^8,y^8,-x^7,y^7,x^6,-y^6,x^5,-y^5,-x^4,y^4,-x^3,y^3,x^2,-y^2,x^1,-y^1)\enspace,
\end{split}
\end{equation}
similarly vefified to be orthonormal.

To see that each vector $\J_{\alpha}N$ is orthogonal to each $\coniugiobaseraw(\blocktext(\J_{\beta})N)$, for $\alpha,\beta=1,\dots,8$, look at Formulas~\eqref{matricesspin7} for $\RO_i,\dots,\RO_h$ and write the octonionic coordinates as $x^\lambda = h_1^\lambda + h_2^\lambda e$, $y^\mu = k_1^\mu + k_2^\mu e$. Then the scalar product $<\J_{\alpha}N,\coniugiobaseraw(\blocktext(\J_{\beta})N)>$ can be computed by using Formula~\eqref{oct} for product of octonions. For example, recall from Formulas~\eqref{B} that
\[
\J_{8}N=(y^1h,x^1h,\dots,y^8h,x^8h,y^9h,x^9h,\dots,y^{16}h,x^{16}h)\enspace,
\]
so that the computation of $<\J_{8}N,\coniugiobaseraw(\blocktext(\J_{1})N)>$ gives rise to pairs of terms like in
\[
<\J_{8}N,\coniugiobaseraw(\blocktext(\J_{1})N)> = \Re(-(\RO_hy^1)\overline{x}^9 -(\RO_h x^9) \overline{y}^1 + \dots) = \Re(\underline{-k k_2^1 \overline{h}_1^9} - \underline{\underline{\overline{h}_2^9 k k_1^1}}-\underline{\underline{k h_2^9\overline{k}_1^1}}- \underline{\overline{k}_2^1 k h_1^9}  + \dots)\enspace.
\]
To conclude, observe that the real part $\Re$ of the sums of each of the corresponding underlined terms is zero. This is due to the identity $\Re(h h' h'') = \Re(h' h'' h)$, that holds for all $h,h',h'' \in \HH$.
\end{proof}

\begin{re} One can recognize that the final argument in the above proof uses the last formula in  \ref{lem:cayley-dickson}. The second formula in the same Lemma is also useful to check orthogonality in higher dimensions, although in Section \ref{sec:generalcase} we will follow a different procedure.
\end{re}

The same argument shows in fact a slightly more general statement:

\begin{pr}
Fix any $\beta$, $1\leq \beta \leq 9$, and consider the $8$ complex structures $\I_{\alpha}\I_{\beta}$, with $\alpha \neq \beta$, defined on  $\RR^{256} =\RR^{16} \oplus \dots \oplus \RR^{16}$ by acting with the corresponding matrices on the listed $16$-dimensional components, that is, by the diagonal extension of $\I_{\alpha}\I_{\beta}$. Consider also the further $8$ complex structures $\coniugiobaseraw(\blocktext(\I_{\alpha}\I_{\beta}))$, for $\alpha \neq \beta$, defined by the same matrices, now acting on the column matrix $(s^1, \dots s^{16})^T$ of sedenions. Then
\[
\{\I_{\alpha}\I_{\beta}N,\coniugiobaseraw(\blocktext(\I_{\alpha}\I_{\beta})N)\}_{\alpha\neq\beta}
\]
is a maximal system of $16$ orthonormal tangent vector fields on $S^{255}$.
\end{pr}


\section{Higher dimensional examples: $S^{511}$ and $S^{4095}$}\label{sec:511}


The dimension $m=2\cdot 16^2$, that is, $S^{511}$, is the lowest case where a last ingredient of our general construction enters the scene. When $q$ increases by $1$ in the decomposition $m=(2k+1)2^p16^q$, in fact, the conjugation is somehow twisted, as we will see in the following.

Imitating what we have done up to now leads to consider the diagonal extensions of $\J_{\alpha}N$ (that is, level $1$ vector fields) and the diagonal extensions of $\coniugiobaseraw(\blocktext(\J_{\alpha})N)$ (that is, level $2$ vector fields), together with one more vector field.

To define this additional vector field we need to extend the formal left multiplication defined by Formula~\eqref{i}. To this aim, consider the decomposition
\[
\RR^{2\cdot 16^2}=\RR^{16^2}\oplus\RR^{16^2}
\]
and denote now by $s^1,s^2$ elements in $\RR^{16^2}$. Then use the formal notation
\begin{align}
N&=(s^1,s^2)\ug s^1+i s^2 \in S^{2\cdot 16^2-1}\enspace,\label{eq:formalCgeneral}
\end{align}
and define a formal left multiplication $\LF_i$ in $\RR^{2\cdot 16^2}$ using Formula~\eqref{i}. One could then expect that $\coniugiobaseraw(\LF_iN)$ be orthogonal to $\{\J_{\alpha}N,\coniugiobaseraw(\blocktext(\J_{\alpha})N)\}_{\alpha=1,\dots,8}$, but this is not the case. In fact, $\coniugiobaseraw(\LF_iN)$ appears to be orthogonal to level $1$ vector fields, but not to level $2$ vector fields. Why this happens, will be clear in Section~\ref{sec:generalcase}.

To make everything work, we need to extend not only $\LF_i$, but also the conjugation $D$. To this aim, split elements $s^\alpha\in\RR^{16^2}$ as $(x^\alpha,y^\alpha)$ where $x^\alpha,y^\alpha\in\RR^{16^2/2}$, and define a conjugation $\coniugiobase{2}$ on $\RR^{16^2}$ using Formula~\eqref{star}:
\begin{equation}\label{eq:coniugiobase2}
\coniugiobase{2}:((x^1,y^1),(x^2,y^2))\longrightarrow((x^1,-y^1),(x^2,-y^2))\enspace.
\end{equation}

The additional vector field we were looking for turns out to be $\coniugiobaseraw(\coniugiobase{2}(\LF_iN))$:
\begin{te}
A maximal orthonormal system of tangent vector fields on $S^{2\cdot 16^2-1}$ is given by the following $8\cdot 2+1$ vector fields:
\begin{equation}
\begin{split}
\J_{1}N\enspace,\enspace&\dots\enspace,\J_{8}N\enspace,\\
\coniugiobaseraw(\blocktext(\J_{1})N)\enspace,\enspace&\dots\enspace,\coniugiobaseraw(\blocktext(\J_{8})N)\enspace,\\
\coniugiobaseraw(\coniugiobase{2}&(\LF_iN))
\end{split}
\end{equation}
\end{te}

\begin{proof}
This statement was first verified through a \mathematica\ computation, like the statements for higher dimensions  entering in the following Remark. For the mathematical proof, we refer to the more general arguments in Section \ref{sec:generalcase}.
\end{proof}

\begin{re}\label{re:4096}
In a completely similar way, we obtain a maximal orthonormal system of tangent vector fields on $S^{2^2 16^2-1}$, $S^{2^3 16^2-1}$ and $S^{16^3-1}$. For instance, the maximal system on $S^{16^3-1}$ is given by:
\begin{equation}
\begin{split}
\J_{1}N\enspace,\enspace&\dots\enspace,\J_{8}N\qquad\text{(level $1$)}\enspace,\\
\coniugiobaseraw(\blocktext(\J_{1})N)\enspace,\enspace&\dots\enspace,\coniugiobaseraw(\blocktext(\J_{8})N)\qquad\text{(level $2$)}\enspace,\\
\coniugiobaseraw(\coniugiobase{2}(\blocktext(\J_{1})N)\enspace,\enspace&\dots\enspace,\coniugiobaseraw(\coniugiobase{2}(\blocktext(\J_{8})N)\qquad\text{(level $3$)}\enspace,
\end{split}
\end{equation}
where diagonal extensions have been used for $\coniugiobaseraw$ and $\coniugiobase{2}$.
\end{re}

\section{The general case}\label{sec:generalcase}

In this Section, we will give a (maximal) orthonormal system of vector fields on any odd-dimensional sphere. To this aim, for any even $m\in\NN$, we identify (linear) vector fields on $S^{m-1}$ with skew-symmetric $m\times m$ matrices, that is, with elements of $\lieso{m}$. The orthogonality condition between vector fields $A,B\in\lieso{m}$ turns then into $AB+BA=0$, and a vector field $A\in\lieso{m}$ has length $1$ if and only if $A^2=-\Idarg{m}$. In this way, we get rid of the normal vector field $N$ used in Sections up to~\ref{sec:511}.

Let $A=(a_{\alpha\beta})_{\alpha,\beta=1,\dots,m}\in\Mat{m}$ be an $m\times m$ matrix, and let $\diag{m}{n},\block{m}{n}:\Mat{m}\rightarrow\Mat{mn}$ be given respectively by
\[
\diag{m}{n}(A)\ug
\begin{pmatrix}
A & & \\
 & \ddots & \\
 & & A
\end{pmatrix}\enspace,\qquad
\block{m}{n}(A)\ug(a_{\alpha\beta}\Idarg{n})_{\alpha,\beta=1,\dots,m}
\]
Thus, $\diag{m}{n}(A)$ is defined by $n^2$ blocks $m\times m$, whereas $\block{m}{n}(A)$ is defined by $m^2$ blocks $n\times n$. The $\diag{m}{n}$ and $\block{m}{n}$ operators formalize the blockwise extension of the action of $A$ on $\RR^m$ to $\RR^{mn}$, seen as $(\RR^m)^n$ and $(\RR^n)^m$ respectively. In particular $\diag{m}{n}$ is what we have called the diagonal extension, and $\block{m}{n}$ is what we have called the block extension in the previous sections.

For instance, if $\J_1$ is the first matrix given in Formula~\eqref{eq:J2}, then
\[
\diag{16}{2}(\J_1)=
\begin{pmatrix}
\J_1 & 0 \\
0 & \J_1
\end{pmatrix}
\enspace\text{and}\enspace
\block{2}{16}
\begin{pmatrix}
0 & -1 \\
1 & 0
\end{pmatrix}
=
\begin{pmatrix}
0 & -\Idarg{16} \\
\Idarg{16} & 0
\end{pmatrix}
\]
are nothing but the matrix of the linear operators $\J_1$ defined as first member in Formula~\eqref{otto'} and $\LF_i$ defined in Formula~\eqref{i} respectively.

\begin{re}
In this section, $\J_1,\dots,\J_8$ denote always the $16$ dimensional matrices given by Formulas~\eqref{eq:J2}.
\end{re}

Observe also that any $A\in\lieso{m}$ induces two vector fields in any dimension multiple of $m$: $\diag{m}{n}(A)$ and $\block{m}{n}$, both in $\lieso{mn}$.

Next Lemma collects properties of $\diagtext$ and $\blocktext$ that we will need.

\begin{lm}\label{lemmadiagblock}
\begin{enumerate}
\item\label{diagblockdiag} If $m\le m'$ and $mn=m'n'$, then $\diag{m}{n}=\diag{m'}{n'}\circ\diag{m}{\frac{m'}{m}}$.
\item\label{diagblockhom} $\diagtext$ and $\blocktext$ are algebra homomorphisms.
\item\label{diagblockim} Let $A=\diag{m}{n}(A')$ and $B=\block{n}{m}(B')$. Then $AB=BA$.
\item\label{diagblockcom} $\diagtext$ and $\blocktext$ commute:
\[
\diag{lm}{n}\circ\block{l}{m}=\block{ln}{m}\circ\diag{l}{n}\enspace.
\]
\end{enumerate}
\end{lm}

\begin{proof}
Proof of points ~\ref{diagblockdiag} and~\ref{diagblockcom} reduce to the definition of $\diagtext$ and $\blocktext$. Point~\ref{diagblockhom} reduces to the fact that sum and multiplication of block matrices respect the block decomposition. As for~\ref{diagblockim}, if $B'=(b_{\alpha\beta})_{\alpha,\beta=1,\dots,n}$, we have:
\[
\begin{split}
AB=
\begin{pmatrix}
A' & & \\
 & \ddots & \\
 & & A'
\end{pmatrix}
\begin{pmatrix}
b_{11}\Idarg{m} & \dots & b_{1n}\Idarg{m} \\
\vdots & \cdots & \vdots \\
b_{n1}\Idarg{m} & \dots & b_{nn}\Idarg{m}
\end{pmatrix}
&=\begin{pmatrix}
A'b_{11}\Idarg{m} & \dots & A'b_{1n}\Idarg{m} \\
\vdots & \cdots & \vdots \\
A'b_{n1}\Idarg{m} & \dots & A'b_{nn}\Idarg{m}
\end{pmatrix}\\
&=\begin{pmatrix}
b_{11}\Idarg{m}A' & \dots & b_{1n}\Idarg{m}A' \\
\vdots & \cdots & \vdots \\
b_{n1}\Idarg{m}A' & \dots & b_{nn}\Idarg{m}A'
\end{pmatrix}=BA
\end{split}
\]

\end{proof}

Remark that for commutation to work in Lemma~\ref{lemmadiagblock}\eqref{diagblockim}, $m$ and $n$ must be swapped.

Using the $\diagtext$ and $\blocktext$ operators, we can now formalize the conjugation used in the previous sections.

\begin{de}
Let $s\in\NN$. Then
\[
\coniugioraw\ug
\begin{pmatrix}
1 & 0 \\
0 & -1
\end{pmatrix}
\in\Mat{2}\enspace,\qquad\coniugiobase{s}\ug\block{2}{\frac{16^s}{2}}(\coniugioraw)\in\Mat{16^s}\enspace.
\]
\end{de}

\begin{re}
The matrix $\coniugiobase{s}$ swaps the signs of the last $\frac{16^s}{2}$ coordinates of a vector in $\RR^{16^s}$. Using this notation, the matrix of the conjugation $\coniugiobaseraw$ defined in Formula~\eqref{star} becomes $\diag{16}{l}(\coniugiobase{1})$, and the matrix of the conjugation $\coniugiobase{2}$ defined in Formula~\eqref{eq:coniugiobase2} becomes $\diag{16^2}{2}(\coniugiobase{2})$.
\hfill\qed
\end{re}

The basic block of our construction is the case $m=16^q$, and this is done in the next subsection (Theorem~\ref{teo:mainq}). The case $m=2^p16^q$ for $p=1,2,3$ is done after the next subsection (Theorem~\ref{teo:mainpq}), and the general case $m=(2k+1)2^p16^q$ will follow (Theorem~\ref{teo:mainkpq}). All the lemmas used in these $3$ theorems are collected at the end of the paper. 

\subsection*{The case $\mathbf{S^{16^q-1}}$, for $\mathbf{q\ge 1}$}

\begin{de}
Let $s,t\in\NN$, where $t\ge 2$ and $s=1,\dots,t-1$. Then
\[
\coniugio{t}{s}\ug\diag{16^s}{16^{t-s}}(\coniugiobase{s})\in\Mat{16^t}\enspace.
\]
\end{de}

With this notation, the cases $q=1,2,3$ described in Sections~\ref{sec:15}, \ref{sec:31-255} and~\ref{sec:511} becomes then:
\begin{itemize}
\item If $q=1$, a maximal system is given by vector fields of level $1$, that is $\{\J_1,\dots,\J_8\}$.
\item If $q=2$, we have vector fields of level $1$ given by $\diag{16}{16}(\{\J_1,\dots,\J_8\})$, and vector fields of level $2$ given by $\coniugio{2}{1}\block{16}{16}(\{\J_1,\dots,\J_8\})$.
\item If $q=3$, we have level $1$ vector fields given by $\diag{16}{16^2}(\{\J_1,\dots,\J_8\})$ and level $2$ vector fields given by $\diag{16^2}{16}(\coniugio{2}{1}\block{16}{16}(\{\J_1,\dots,\J_8\}))$. Moreover, we have $8$ further level $3$ vector fields given by $\coniugio{3}{2}\coniugio{3}{1}\block{16}{16^2}(\{\J_1,\dots,\J_8\})$.
\end{itemize}

Denoting the product of conjugations by
\begin{equation}\label{eq:coniugiototale}
\Mat{16^t}\ni\coniugiototale{t}\ug
\begin{cases}
\Idarg{16} & \text{if } t=1\enspace,\\
\prod_{s=1}^{t-1}\coniugio{t}{s} & \text{if } t\ge 2\enspace,
\end{cases}
\end{equation}
we can state the general theorem for $S^{16^q-1}$.

\begin{te}\label{teo:mainq}
For any $q\ge 1$, the $8q$ vector fields on $S^{16^q-1}$ given by
\[
\{B^q(t,\J_\alpha)\ug\diag{16^t}{16^{q-t}}(\coniugiototale{t}\block{16}{16^{t-1}}(\J_\alpha))\}_{\substack{t=1,\dots,q\\ \alpha=1,\dots,8}}
\]
are a maximal orthonormal set.
\end{te}

\begin{proof}
If $q=1$, then $B^q(t,\J_\alpha)=\J_\alpha$ and the statement reduces to Proposition~\ref{otto}. Thus, assume that $q\ge 2$.

Choose vector fields $B^q(t,\J), B^q(t',\J')$, and assume with no loss of generality that $t\le t'$. Then, observe that
\[
\begin{split}
B^q(t,\J)B^q(t',\J')&=\diag{16^t}{16^{q-t}}(\coniugiototale{t}\block{16}{16^{t-1}}(\J))\diag{16^{t'}}{16^{q-t'}}(\coniugiototale{t'}\block{16}{16^{t'-1}}(\J'))\\
&\shortstack{{\tiny\ref{lemmadiagblock}\eqref{diagblockdiag}}\\=}\diag{16^{t'}}{16^{q-t'}}(\diag{16^t}{16^{t'-t}}(\coniugiototale{t}\block{16}{16^{t-1}}(\J))\coniugiototale{t'}\block{16}{16^{t'-1}}(\J'))\enspace.
\end{split}
\]
Thus, it is enough to consider the case $t'=q$. We are then reduced to show that, for any $q\ge 2$ and $\J,\J'\in\{\J_1,\dots,\J_8\}$:
\begin{enumerate}
\item $B^q(t,\J)B^q(q,\J')+B^q(q,\J')B^q(t,\J)=0$ for $1\le t\le q$ and $(t,\J)\neq(q,\J')$;
\item $B^q(q,\J)$ is an almost complex structure.
\end{enumerate}

We divide the proof of (1) in the cases $1\le t\le q-1$ and $t=q$.

If $1\le t\le q-1$, we have
\[
\begin{split}
B^q(t,\J)B^q(q,\J')&=B^q(t,\J)\coniugiototale{q}\block{16}{16^{q-1}}(\J')\\
&=B^q(t,\J)\prod_{s=1}^{q-1}\coniugio{q}{s}\block{16}{16^{q-1}}(\J')\\
&\shortstack{{\tiny\ref{lemmadue}}\\=}\prod_{\substack{s=1\\s\neq t}}^{q-1}\coniugio{q}{s}B^q(t,\J)\coniugio{q}{t}\block{16}{16^{q-1}}(\J')\\
&\shortstack{{\tiny\ref{lemmatre}}\\=}-\prod_{s=1}^{q-1}\coniugio{q}{s}B^q(t,\J)\block{16}{16^{q-1}}(\J')\\
&\shortstack{{\tiny\ref{lemmauno}}\\=}-\prod_{s=1}^{q-1}\coniugio{q}{s}\block{16}{16^{q-1}}(\J')B^q(t,\J)=-B^q(q,\J')B^q(t,\J)\enspace.
\end{split}
\]

If $t=q$ and $\J\neq \J'$, we have
\[
\begin{split}
B^q(q,\J)B^q(q,\J')=&\coniugiototale{q}\block{16}{16^{q-1}}(\J)\coniugiototale{q}\block{16}{16^{q-1}}(\J')\\
\shortstack{{\tiny\ref{lemmaquattro}}\\=}&\coniugiototale{q}\coniugiototale{q}\block{16}{16^{q-1}}(\J)\block{16}{16^{q-1}}(\J')\\
\shortstack{{\tiny\ref{lemmadiagblock}\eqref{diagblockhom}}\\=}&\coniugiototale{q}\coniugiototale{q}\block{16}{16^{q-1}}(\J\J')\\
\shortstack{{\tiny $\J\perp \J'$}\\=}&\coniugiototale{q}\coniugiototale{q}\block{16}{16^{q-1}}(-\J'\J)\\
\shortstack{{\tiny\ref{lemmadiagblock}\eqref{diagblockhom}}\\=}-&\coniugiototale{q}\block{16}{16^{q-1}}(\J')\coniugiototale{q}\block{16}{16^{q-1}}(\J)=-B^q(q,\J')B^q(q,\J)\enspace.
\end{split}
\]

Finally, the case $t=q$ and $\J=\J'$. We have to show that $B^q(q,\J)$ is an almost complex structure. We have
\[
\begin{split}
B^q(q,\J)B^q(q,\J)&=\coniugiototale{q}\block{16}{16^{q-1}}(\J)\coniugiototale{q}\block{16}{16^{q-1}}(\J)\\
&\shortstack{{\tiny\ref{lemmaquattro}\ref{lemmadiagblock}\eqref{diagblockhom}}\\=}\coniugiototale{q}\coniugiototale{q}\block{16}{16^{q-1}}(\J^2)\shortstack{{\tiny\ref{lemmacinque}}\\=}-\Idarg{16^q}\enspace.
\end{split}
\]
\end{proof}

\begin{re}
From the proof, it appears that for any $t=1,\dots,q-1$, the level $t$ conjugation $\coniugio{q}{t}$ makes level $q$ vector fields $B(q,\J_\alpha)$ orthogonal to level $t$ vector fields $B(t,\J_\beta)$.
\hfill\qed
\end{re}

\begin{re}\label{re:overload}
Overloading the symbols $\coniugiototale{t}$ and $\coniugioraw$, the linear operators used in Table~\ref{table:generalcase} in the Introduction can now be defined by
\[
\begin{split}
\coniugiototale{t}:\Mat{16}&\longrightarrow\Mat{(2k+1)2^p16^q}\\
A&\longmapsto\diag{16^t}{(2k+1)2^p16^{q-t}}(\coniugiototale{t}\block{16}{16^{t-1}}(A))
\end{split}
\]
and by
\[
\begin{split}
\coniugioraw:\Mat{2^p}&\longrightarrow\Mat{(2k+1)2^p16^q}\\
A&\longmapsto\diag{2^p16^q}{2k+1}(\diag{16^q}{2^p}(\coniugiototale{q})\block{2^p}{16^q}(A))\enspace.
\end{split}
\]
\hfill\qed
\end{re}

\subsection*{The case $\mathbf{S^{2^p16^q}}$, for $\mathbf{p=1,2,3}$}

When $p=1,2$ or $3$, we have $1,3$ or $7$ additional vector fields, given essentially by the complex, quaternionic or octonionic multiplication $\LC,\LH$ or $\LO$ respectively. Thus, to state the Theorem we give the following definition.

\begin{de}
\[
\mathcal{G}^1\ug\{\LC_i\}\subset\Mat{2}\enspace,\qquad\mathcal{G}^2\ug\{\LH_i,\LH_j,\LH_k\}\subset\Mat{4}\enspace,\qquad\mathcal{G}^3\ug\{\LO_i,\LO_j,\LO_k,\LO_e,\LO_f,\LO_g,\LO_h\}\subset\Mat{8}\enspace.
\]
\end{de}


\begin{te}\label{teo:mainpq}
For any $q\ge 1$ and $p=1,2$ or $3$, the $8q+2^p-1$ vector fields on $S^{2^p16^q-1}$ given by
\[
\begin{split}
\{B^{p,q}(t,\J_\alpha)&\ug\diag{16^q}{2^p}(B^q(t,\J_\alpha))\}_{\substack{t=1,\dots,q\\ \alpha=1,\dots,8}}\\
\{\LF^{p,q}(G)&\ug\diag{16^q}{2^p}(\coniugiototale{q}\coniugiobase{q})\block{2^p}{16^q}(G)\}_{G\in\mathcal{G}^p}
\end{split}
\]
are a maximal orthonormal system.
\end{te}

\begin{proof}
The orthonormality for $\{\LF^{p,q}(G)\}_{G\in\mathcal{G}^p}$ is a direct consequence of Lemma~\ref{lemmadiagblock}\eqref{diagblockim}\eqref{diagblockhom}, the orthonormality of $\mathcal{G}^p$, and the fact that $\diag{16^q}{2^p}(\coniugiototale{q}\coniugiobase{q})\diag{16^q}{2^p}(\coniugiototale{q}\coniugiobase{q})=\Idarg{2^p16^q}$.

The orthonormality for $\{B^{p,q}(t,\J_\alpha)\}_{\substack{t=1,\dots,q\\ \alpha=1,\dots,8}}$ follows from Theorem~\ref{teo:mainq}.

We are then left to show that $B^{p,q}(t,\J_\alpha)\LF^{p,q}(G)+\LF^{p,q}(G)B^{p,q}(t,\J_\alpha)=0$, for $t=1,\dots,q$ and $\alpha=1,\dots,8$. But
\[
\begin{split}
B^{p,q}(t,\J_\alpha)\LF^{p,q}(G)&=\diag{16^q}{2^p}(B^q(t,\J_\alpha))\diag{16^q}{2^p}(\coniugiototale{q}\coniugiobase{q})\block{2^p}{16^q}(G)\\
&\shortstack{{\tiny\ref{lemmadiagblock}\eqref{diagblockhom}}\\=}\diag{16^q}{2^p}(B^q(t,\J_\alpha)\coniugiototale{q}\coniugiobase{q})\block{2^p}{16^q}(G)\\
&\shortstack{{\tiny\ref{lem:lemmapq}}\\=}-\diag{16^q}{2^p}(\coniugiototale{q}\coniugiobase{q}B^q(t,\J_\alpha))\block{2^p}{16^q}(G)\\
&\shortstack{{\tiny\ref{lemmadiagblock}\eqref{diagblockhom}}\\=}-\diag{16^q}{2^p}(\coniugiototale{q}\coniugiobase{q})\diag{16^q}{2^p}(B^q(t,\J_\alpha))\block{2^p}{16^q}(G)\\
&\shortstack{{\tiny\ref{lemmadiagblock}\eqref{diagblockim}}\\=}-\diag{16^q}{2^p}(\coniugiototale{q}\coniugiobase{q})\block{2^p}{16^q}(G)\diag{16^q}{2^p}(B^q(t,\J_\alpha))=-\LF^{p,q}(G)B^{p,q}(t,\J_\alpha)\enspace.
\end{split}
\]
\end{proof}

\subsection*{The general case: $\mathbf{S^{m-1}}$ for any even $\mathbf{m}$}

Defining $\mathcal{G}^0\ug\emptyset$, we can state the general case $S^{m-1}$, $m=(2k+1)2^p16^q$, in one single Theorem:
\begin{te}\label{teo:mainkpq}
For any $k\ge 0$, $q\ge 1$ and $p=0,1,2$ or $3$, the $8q+2^p-1$ vector fields on $S^{(2k+1)2^p16^q-1}$ given by
\[
\begin{split}
\{B^{k,p,q}(t,\J_\alpha)&\ug\diag{16^t}{(2k+1)2^p16^{q-t}}(\coniugiototale{t}\block{16}{16^{t-1}}(\J_\alpha))\}_{\substack{t=1,\dots,q\\ \alpha=1,\dots,8}}\\
\{\LF^{k,p,q}(G)&\ug\diag{2^p16^q}{2k+1}(\diag{16^q}{2^p}(\coniugiototale{q})\block{2^p}{16^q}(G))\}_{G\in\mathcal{G}^p}
\end{split}
\]
are a maximal orthonormal set.
\end{te}

\begin{proof}
Follows from Theorems~\ref{teo:mainq}, \ref{teo:mainpq} and property \eqref{diagblockhom} of $\diagtext$ in Lemma~\ref{lemmadiagblock}.
\end{proof}

\subsection*{Lemmas}

In this last subsection, we collect all the lemmas appearing as references in the equalities to prove Theorems~\ref{teo:mainq}, \ref{teo:mainpq} and~\ref{teo:mainkpq}.
\begin{lm}\label{lemmauno}
Let $q\ge 2$ and $1\le t\le q-1$. Then $B^q(t,J)$ and $\block{16}{16^{q-1}}(J')$ commute.
\end{lm}

\begin{proof}
Follows from the fact that $B^q(t,J)\in\II(\diag{16^{q-1}}{16})$ (Lemma~\ref{lemmadiagblock}\eqref{diagblockdiag}) and Lemma~\ref{lemmadiagblock}(3).
\end{proof}

\begin{lm}\label{lemmaquattro}
Let $q\ge 2$ and $1\le t\le q-1$. Then $\coniugio{q}{t}$ and $\block{16}{16^{q-1}}(J)$ commute.
\end{lm}

\begin{proof}
Follows from the fact that $\coniugio{q}{t}\in\II(\diag{16^{q-1}}{16})$ (Lemma~\ref{lemmadiagblock}\eqref{diagblockdiag}) and Lemma~\ref{lemmadiagblock}\eqref{diagblockim}.
\end{proof}

\begin{lm}\label{lemmadue}
Let $q\ge 2$, $1\le t,s\le q-1$ and $s\neq t$. Then $\coniugio{q}{s}$ and $B^q(t,\J)$ commute.
\end{lm}

\begin{proof}
Assume $s>t$. Then the claim follows from Lemma~\ref{lemmadiagblock}\eqref{diagblockim}, writing $\coniugio{q}{s}\shortstack{{\tiny\ref{lemmadiagblock}\eqref{diagblockcom}}\\=}\block{2\cdot 16^{q-s}}{\frac{16^s}{2}}(\diag{2}{16^{q-s}}(C))$ and $B^q(t,\J)\in\II(\diag{\frac{16^s}{2}}{2\cdot 16^{q-s}})$ (Lemma~\ref{lemmadiagblock}\eqref{diagblockdiag}).

If $s<t$, we use Lemma~\ref{lemmadiagblock}\eqref{diagblockdiag} to write $\coniugio{q}{s}=\diag{16^t}{16^{q-t}}(\coniugio{t}{s})$, so that
\[
\begin{split}
\coniugio{q}{s}B^q(t,\J)&\shortstack{{\tiny\ref{lemmadiagblock}\eqref{diagblockhom}}\\=}\diag{16^t}{16^{q-t}}(\coniugio{t}{s}\coniugiototale{t}\block{16}{16^{t-1}}(\J))\\
&=\diag{16^t}{16^{q-t}}(\coniugiototale{t}\coniugio{t}{s}\block{16}{16^{t-1}}(\J))\\
&\shortstack{{\tiny\ref{lemmaquattro}}\\=}\diag{16^t}{16^{q-t}}(\coniugiototale{t}\block{16}{16^{t-1}}(\J)\coniugio{t}{s})\\
&=B^q(t,\J)\coniugio{q}{s}\enspace,
\end{split}
\]
where $\coniugiototale{t}$ and $\coniugio{t}{s}$ commutes since they are both diagonal matrices.
\end{proof}

\begin{lm}\label{lemmatre}
Let $q\ge 2$ and $1\le t\le q$. Then $\coniugio{q}{t}$ and $B^q(t,\J)$ anticommute.
\end{lm}

\begin{proof}
\[
\begin{split}
\coniugio{q}{t}B^q(t,\J)&\shortstack{{\tiny\ref{lemmadiagblock}\eqref{diagblockhom}}\\=}\diag{16^t}{16^{q-t}}(\coniugiobase{t}\coniugiototale{t}\block{16}{16^{t-1}}(\J))\\
&=\diag{16^t}{16^{q-t}}(\coniugiototale{t}\coniugiobase{t}\block{16}{16^{t-1}}(\J))
\end{split}
\]
where $\coniugiototale{t}$ and $\coniugiobase{t}$ commutes since they are both diagonal matrices. We are thus reduced to show that $\coniugiobase{t}$ and $\block{16}{16^{t-1}}(\J)$ anticommutes. But using the explicit expressions given in Formulas~\ref{eq:J2} for the $\J$s, we can write $\block{16}{16^{t-1}}(\J)$ as
\[
\block{16}{16^{t-1}}(\J)=
\begin{cases}
\begin{pmatrix}
0 & \block{8}{16^{t-1}}(\RO) \\
\block{8}{16^{t-1}}(\RO) & 0
\end{pmatrix}
&
\text{if }\J\in\{\J_2,\dots,\J_8\}\enspace,\\
\begin{pmatrix}
0 & -\Idarg{\frac{16^t}{2}} \\
\Idarg{\frac{16^t}{2}} & 0
\end{pmatrix}
&
\text{if }\J=\J_1\enspace,
\end{cases}
\]
and in both cases a block multiplication by hand shows that $\coniugiobase{t}\block{16}{16^{t-1}}(\J)=-\block{16}{16^{t-1}}(\J)\coniugiobase{t}$.
\end{proof}

\begin{co}\label{co:coniugiototale}
Let $q\ge 2$ and $1\le t\le q-1$. Then $\coniugiototale{q}$ and $B^q(t,\J)$ anticommute.
\end{co}

\begin{proof}
Follows from the Definition~\eqref{eq:coniugiototale} of $\coniugiototale{q}$ and Lemmas~\ref{lemmadue},~\ref{lemmatre}.
\end{proof}

\begin{lm}\label{lemmacinque}
Let $q\ge 2$ and $1\le t\le q-1$. Then $\coniugio{q}{t}\coniugio{q}{t}=\Idarg{16^q}$.
\end{lm}

\begin{proof}
Follows from Lemma~\ref{lemmadiagblock}\eqref{diagblockhom} and from $C^2=\Idarg{2}$.
\end{proof}

\begin{lm}\label{lem:lemmapq}
Let $q\ge 1$ and $1\le t\le q$. Then $B^q(t,\J)$ and $\coniugiototale{q}\coniugiobase{q}$ anticommute.
\end{lm}

\begin{proof}
If $q=1$, then $B^q(t,\J)=\J$ and $\coniugiototale{q}\coniugiobase{q}=\coniugiobase{1}$ anticommute, as in proof of Lemma~\ref{lemmatre} with $t=1$.

If $q\ge 2$, then split the proof in cases $1\le t\le q-1$ and $t=q$.

If $q\ge 2$ and $1\le t\le q-1$, use Corollary~\ref{co:coniugiototale} to show that $B^q(t,\J)$ and $\coniugiototale{q}$ anticommute. Then, write $B^q(t,\J)\in\II(\diag{\frac{16^q}{2}}{2})$ and use Lemma~\ref{lemmadiagblock}\eqref{diagblockim} to show that $B^q(t,\J)$ and $\coniugiobase{q}$ commute.

If $q\ge 2$ and $t=q$, use Lemma~\ref{lemmaquattro} to show that $B^q(q,\J)$ and $\coniugiototale{q}$ commute, and finally apply Lemma~\ref{lemmatre} with $t=q$ to show that $B^q(q,\J)$ and $\coniugiobase{q}$ anticommute.
\end{proof}

\begin{re}
It would be interesting, but we were not able, to compare the maximal systems of vector fields constructed in the present paper with the ones appearing in previous constructions, in particular with the vector fields obtained in \cite{og}. 
\end{re}


\begin{thebibliography}{99}

\bibitem {ad1} J. F. Adams, \emph{Vector fields on spheres}, Ann. of Math. {\bf 75} (1962), 603-632. 
\bibitem {ad2} J. F. Adams, \emph{Vector fields on spheres}, Topology {\bf 1} (1962), 63-65. 
\bibitem {ad3} J. F. Adams, \emph{Vector fields on spheres}, Bull. Amer. Math. Soc. {\bf 68} (1962), 39-41. 
\bibitem {af} C. S. Aravinda, F. T Farrel, \emph{Exotic negatively curved structures on Cayley hyperbolic manifolds}, J. Diff. Geom. {\bf 63} (2003), 41-62.
\bibitem{baez} J. C. Baez, \emph{The Octonions},  Bull. Amer. Math. Soc. {\bf 39}  (2002),  145--205;  {\bf 42}  (2005),  213.
\bibitem{bdi} D. K. Biss, D. Dugger, D. C. Isaksen, \emph{Large Annihilators in Cayley-Dickson Algebras}, Comm. in Algebra {\bf 36} (2008), 632-664; (also with J. D. Christensen), Part II, Bol. Soc. Mat. Mexicana {\bf 13} (2007),  269--292.
\bibitem{bcdi} D. K. Biss, J. D. Christensen, D. Dugger, D. C. Isaksen,  \emph{Eigentheory of Cayley-Dickson algebras}, Forum Math. {\bf 21} (2009), 833--851.
\bibitem{bss} M. Bre\v{s}ar, P. \v{S}emrl, \v{S}. \v{S}penko, \emph{On locally complex algebras and low-dimensional Cayley-Dickson algebras}, J. of Algebra {\bf 327} (2011), 107-125.
\bibitem{br-gr} R. B. Brown, A. Gray, \emph{Riemannian Manifolds with Holonomy ${\mathrm Spin}(9)$}, Diff. Geometry in honour of K. Yano, Kinokuniya, Tokyo (1972), 41-59.
\bibitem{ec} B. Eckmann, \emph{Gruppentheoretischer Beweis des Satzes von Hurwitz-Radon \"uber die Komposition quadratischer Formen.}, Comment. Math. Helv. {\bf 15} (1943), 362-366.
\bibitem{fr} Th. Friedrich, \emph{Weak $\Spin{9}$-Structures on 16-dimensional Riemannian Manifolds},  
Asian J. Math. {\bf 5}  (2001),  129--160. 
\bibitem{gl-wa-zi} H. Gluck, F. Warner, W. Ziller, \emph{The geometry of the Hopf fibrations}, L'Enseignement Math. {\bf 32}
(1986), 173-198.
\bibitem{ha} R. Harvey, Spinors and Calibrations, Academic Press, 1990.
\bibitem{ha-la} R. Harvey, H. B. Lawson Jr., \emph{Calibrated Geometries}, Acta Math. {\bf 148} (1982), 47-157. 
\bibitem{hur} A. Hurwitz, \emph{\"Uber die Komposition der quadratischen Formen}, Math. Ann. {\bf 88} (1922), 1-25; reproduced in \emph{Mathematische Werke: Zahlentheorie, Algebra und Geometrie,.Band II}, Birkh\"auser (1963), 641-666.
\bibitem{hu} D. Husemoller, \emph{Fibre Bundles}, 2nd edition, Springer-Verlag, 1975. 
\bibitem{ka} M. Karoubi, \emph{K-Theory: an Introduction}, reprint of the 1978 edition, Springer-Verlag, 2008.
\bibitem{lv} B. Loo, A. Verjovsky, \emph{The Hopf fibration over $S^8$ admits no $S^1$-subfibrations}, Topology {\bf
31} (1992), 239-254.
\bibitem{og} A. A.Ognikyan, \emph{Combinatorial Construction of Tangent Vector Fields on Spheres}, Math. Notes {\bf 83} (2008), 590-605.
\bibitem{pp} M. Parton, P. Piccinni, \emph{$\mathrm{Spin}(9)$ and almost complex structures on $16$-dimensional manifolds}, Ann. Global An. Geom., 41 (2012), 321-345.
\bibitem{ra} J. Radon, \emph{Lineare Scharen orthogonaler Matrizen}, Abh. Math. Sem. Univ. Hamburg {\bf 1} (1922), 1-14.
\bibitem{ds}  D. A. Salamon, Th. Walpuski, \emph{Notes on the Octonians}, arXiv:1005.2820 (2010), 1-73.
\bibitem{th}  E. Thomas, \emph{Vector fields on manifolds}, Bull. Amer. Math. Soc. {\bf 75}, (1969), 643-683. 

\end{thebibliography}
\end{document}